\documentclass[none]{apjrnl}
\usepackage[dvips]{graphicx}
\usepackage{amsfonts}
\usepackage{amssymb}
\usepackage{amsmath}
\usepackage{latexsym}
\usepackage[finnish,english]{babel}
\usepackage[latin1]{inputenc}

{\theoremstyle{plain}%
  \newtheorem{theorem}{Theorem}
  \newtheorem{corollary}{Corollary}
  \newtheorem{proposition}{Proposition}
  \newtheorem{lemma}{Lemma}%

{\theoremstyle{remark}

}
{\theoremstyle{definition}
\newtheorem{definition}{Definition}
\newtheorem{example}{Example}
}

\begin{document}

\newcommand{\dunion}{\amalg}

\newcommand{\E}{{\mathcal E}}
\newcommand{\Z}{{\mathcal Z}}
\newcommand{\A}{{\mathcal A}}

\title{The Ring of Graph Invariants - Graphic Values}
\author{Tomi Mikkonen\\ \small{tomi.mikkonen@cs.tut.fi}}
\date{30th November 2007} 

\maketitle

\begin{abstract}
The ring of graph invariants is spanned by the basic graph invariants which
calculate the number of subgraphs isomorphic to a given graph in other graphs.
Sets of basic graph invariants form $G$-posets where each graph in the set
induces the corresponding invariant calculating the number of subgraphs
isomorphic to this graph in other graphs. It is well known that all other
graph invariants such as sorted eigenvalues and canonical permutations are
linear combinations of the basic graph invariants. 

These subgraphs counting invariants are not algebraically independent. In our
view the most important problem in graph theory of unlabeled graphs is the
problem of determining graphic values of arbitrary sets of graph
invariants. This corresponds to explaining the syzygy of the graph invariants
when the number of vertices is unbounded. We introduce two methods to explore
this complicated structure.

$G$-posets with a small number of vertices impose constraints on larger
$G$-posets. We describe families of inequalities of graph
invariants. These inequalities allow to loop over all values of graph
invariants which look like graphic from the small $G$-posets point of
view. The inequalities give rise to a weak notion of graphic values where the
existence of the corresponding graph is not guaranteed.

We also develop strong notion of graphic values where the existence of the
corresponding graphs is guaranteed once the constraints are satisfied by the
basic graph invariants. These constraints are necessary and sufficient for
graphs whose local neighborhoods are generated by a finite set of locally
connected graphs. The reconstruction of the graph from the basic graph
invariants is shown to be NP-complete in similarly restricted case.

Finally we apply these results to formulate the problem of Ramsey numbers as
 an integer polyhedron problem of moderate and adjustable dimension. 
\end{abstract}

\section{Introduction}
In this paper we study \emph{basic graph invariants} which count the number of
subgraphs isomorphic to $g$ in $h$, see \cite{Tomi3} for more complete
introduction. We denote by $I(g)(h)$ the number of subgraphs isomorphic to $g$
in the graph $h$. For simple graphs we use monomial
notation in $\mathbb{C}[a_{ij}]$, such that the monomial $\prod_{(i,j)\in
  E}a_{ij}$ represents the graph $(V,E)$.

For example  $I(a_{12})(a_{12}a_{23}a_{34}a_{14}a_{13})=5$ and $I(a_{12}a_{13})(a_{12}a_{23}a_{34}a_{14}a_{13})=8$. This definition does not depend on the labeling of the graphs $g$ and $h$ but only on their isomorphism classes.

Let $A$ be the adjacency matrix of the graph $h$, i.e. $a_{ij}=1$ if there is
an edge between the vertices $i$ and $j$ in $h$ and $a_{ij}=0$
otherwise. Because $h$ is a simple graph we have $a_{ij}=a_{ji}$. Let
the monomial $a_{i_1j_1}a_{i_2j_2}\cdots a_{i_dj_d}$ have the structure of $g$
i.e. the monomial contains all the variables corresponding to the edges of
$g$. Then $I(g)(h)$ is a function in the variables $a_{ij}$ and can be written explicitly as 
\begin{equation}\label{eq:eqeka}
I(g)(h)=\sum_{\rho \in S_n/\mathrm{Stab}(a_{i_1j_1}a_{i_1j_2}\cdots a_{i_dj_d})} a_{\rho(i_1)\rho(j_1)}a_{\rho(i_2)\rho(j_2)}\cdots a_{\rho(i_d)\rho(j_d)}.
\end{equation}
 We call the polynomials $I(g)(h)$ \textit{basic graph invariants of type
 $g$}. We may consider the basic graph invariants as symbolic polynomials in
 $a_{ij}$ and we often drop the second graph ($h$ in $I(g)(h)$) from the
 notation. We use the notation $I(g)$ for this symbolic polynomial, where $g$
 is some monomial in the orbit sum. 

The basic graph invariants are not algebraically independent. Let $cv(g)$
denote the number of vertices in $g$ having at least one edge connected to it
and let $|g|$ be the number edges in $g$. The product of
two basic graph invariants is the following linear combination of the basic
graph invariants:
\begin{equation}\label{eq:29}
I(g_i)I(g_j)=\sum_{k=1}^N \left( \sum_{h=1}^N (-1)^{|g_k|-|g_h|}e_{kh}e_{hi}e_{hj} \right) I(g_k),
\end{equation}
where $e_{ij}=I(g_j)(g_i)$ and the set $\{g_1,g_2,\ldots,g_N\}$ contains all
graphs with $cv(g_i)+cv(g_j)$ vertices and less. The equation remains valid if
we add all graphs to the set. This formula is originally due to V. L Mnukhin. There is also Fleischmann's product formula which gives the coefficients in the expansion \cite{Fleischmann}. In \cite{Tomi3} we study the minimal generators of the ring
of graph invariants by using these formulae.

To fully understand the structure of the ring of graph invariants it became
clear that we need a notion of \emph{$G$-posets}.

\begin{definition}
$G$-poset is a pair $(\E,G)$, where $\E$ is the set of
monomial equivalence classes/invariants with respect to $G$ in $R[x_1,\ldots,x_N]^G$ and $G\subseteq S_N$ is the permutation group acting
on $R[x_1,\ldots,x_N]$.
\end{definition}

This notion is intended to stress the sociological behavior of the monomials which
means that each monomial corresponds to a basic invariant which can
be evaluated in all other monomials. In this paper we restrict to multilinear
monomials which suffice for simple graphs. It is, however,  possible to
generalize this notion to general monomials \cite{Tomi3}.

We say that a set of monomial equivalence classes $\E$ is a \textit{complete $G$-poset} with respect to the permutation group $G$ if the following holds. 
\begin{itemize}
\item[] For all monomials $w$ appearing in the orbit sums of the invariants, all the submonomials $m\subseteq w$ appear also in some orbit sum in the $G$-poset.
\end{itemize}

For each $G$-poset $\E$ there is the corresponding $E$-\emph{transform} of $\E$
as a matrix with entries $e_{ij}=I(m_j)(m_i)$, where $m_i$, $i=1\ldots N$ are
all the monomials representing the orbit sums in the $G$-poset $\E$. 

We denote by $\E(n)$ the $G$-poset of simple graphs with $n$ vertices and by
$\E(n,d)$ we denote the $G$-poset of simple graphs with $n$ vertices and at
most $d$ edges.

\begin{example}\label{ex:aaa}
Consider the $G$-poset $\E(4)$ with the basic graph invariants $g_0=I(\emptyset)=1$, $g_1=I(a_{12})$, $g_2=I(a_{12}a_{34})$, $g_3=I(a_{12}a_{13})$, $g_4=I(a_{12}a_{23}a_{34})$, $g_5=I(a_{12}a_{13}a_{14})$, $g_6=I(a_{12}a_{13}a_{23})$, $g_7=I(a_{12}a_{23}a_{34}a_{14})$, $g_8=I(a_{12}a_{23}a_{24}a_{34})$, $g_9=I(a_{12}a_{23}a_{34}a_{14}a_{13})$, $g_{10}=I(a_{12}a_{23}a_{34}a_{14}a_{13}a_{24})$. The $E$-transform is
\begin{displaymath}
E=\left( \begin{array}{ccccccccccc}
1 & 0 & 0 & 0 & 0 & 0 & 0 & 0 & 0 & 0 & 0 \\
1 & 1 & 0 & 0 & 0 & 0 & 0 & 0 & 0 & 0 & 0 \\
1 & 2 & 1 & 0 & 0 & 0 & 0 & 0 & 0 & 0 & 0 \\
1 & 2 & 0 & 1 & 0 & 0 & 0 & 0 & 0 & 0 & 0 \\
1 & 3 & 0 & 3 & 1 & 0 & 0 & 0 & 0 & 0 & 0 \\
1 & 3 & 1 & 2 & 0 & 1 & 0 & 0 & 0 & 0 & 0 \\
1 & 3 & 0 & 3 & 0 & 0 & 1 & 0 & 0 & 0 & 0\\ 
1 & 4 & 1 & 5 & 1 & 2 & 1 & 1 & 0 & 0 & 0 \\
1 & 4 & 2 & 4 & 0 & 4 & 0 & 0 & 1 & 0 & 0 \\
1 & 5 & 2 & 8 & 2 & 6 & 2 & 4 & 1 & 1 & 0\\ 
1 & 6 & 3 & 12 & 4 & 12 & 4 & 12 & 3 & 6 & 1 \\
\end{array} \right).
\end{displaymath}
We write indices from $0$ to $10$. Thus for example $e_{9,3}=I(g_3)(g_9)=8$. 
\end{example}

The inverse $B=E^{-1}$ is easy to calculate when the $G$-poset is complete; it is defined by the elements
$b_{ij}=(-1)^{|g_i|-|g_j|}e_{ij}$, where $e_{ij}$ is the element of the
$E$. The theory of $G$-posets remains valid for arbitrary permutation groups $G$
and we also exploit this in defining the local invariants. The $E$-transform
is vital in understanding the constraints for the basic graph invariants.  

This paper is divided three sections. Section $2$ is devoted to study the
necessary constraints for the values of basic graph invariants. 

Section $3$ describes necessary and sufficient conditions for the graphic values. These
necessary and sufficient conditions apply only if we restrict to locally
finitely generated graphs. Also more accurate necessary constraints are developed for the general case.

In section $4$ we apply these results to the Ramsey numbers and show certain
invariants which are inevitably related to cliques and Ramsey invariants. In
general any extremal graph problem is of form: For a given property, generate
the graph having this property. Since all properties can be expressed as
linear combinations of the basic graph invariants and we now have results which
describe the graphic values of the basic graph invariants, we can solve all
the graph extremal problems in principle. Some properties of graphs may be of
high degree in the representation as a linear combination of the basic graph
invariants thus spoiling this attempt. 

Ramsey numbers require only low degree representation. The Ramsey number $r(k)$ is
the minimum number of vertices $n$ such that all undirected simple graphs of
order $n$ contain a clique of order $k$ or an 
independent set of order $k$. We use notation $I(K_k)$ for the number of
cliques and $I({\overline K}_k)$ for the number of independent sets. Let us
call the sum $I(K_k)+I({\overline K}_k)$ the \emph{Ramsey invariant}.
In this language the Ramsey number $r(k)$ is the smallest $n$ s.t.
\begin{equation}
I(K_k)(H)+I({\overline K}_k)(H) \geq 1 \ \ \forall\ H \in \E(n).
\end{equation}
Our approach is to find a lower bound for this graph invariant. We prove in Section \ref{sec:ram} that the invariant $I({\overline K}_k)$ can be written in terms of basic graph invariants as
\begin{equation}
I({\overline K}_k)=\sum_{A \subseteq K_k} (-1)^{|A|}{n-cv(A) \choose k-cv(A)} I(A),
\end{equation}
where $cv(A)$ denotes the number of vertices connected to the edges of the graph $A$ and $|A|$ denotes the number of edges in the graph $A$. The sum is over unlabeled subgraphs of the complete graph $K_k$.

The problem in finding the Ramsey numbers is that the $G$-poset $\E(n)$, in
which the lower bound for the Ramsey invariant should be calculated, is very
large. For instance $r(4)=18$ meaning that we should test the inequality for
all the graphs in the $G$-poset $\E(18)$ of size
$1787577725145611700547878190848$. However our results show that $\E(4)$ 
is sufficient for proving $r(3)=6$. Thus we can hope to find the value
of $r(5)$ using a much smaller $G$-poset than $\E(43)$ which is currently the
best lower bound for $r(5)$.

There are also numerous results on the graded algebra $(\mathbb{C}[a_{ij}]/\langle a_{ij}-a_{ji} \rangle)^{S_n}$
studied more or less by the tools of classical invariant theory \cite{Thiery},\cite{Bondy},\cite{Kocay},\cite{Pouzet}.

\section{Weakly Graphic Values}
Let $\E(n)$ denote the $G$-poset of graphs with $n$ vertices. In this section
we find a set of constraints for the large $G$-poset $\E(n)$ from the smaller
$G$-poset $\E(r)$, where $n\geq r$. 

Let $z$ denote the vector $[1,I(g_1),I(g_2),$ $\ldots$, $I(g_{m})]$ consisting of all basic graph invariants of the $G$-poset $\E(r)$. The vector $z$ can have only certain values in the $G$-poset $\E(n)$. It turns out that when $z$ is evaluated on the graph of the $G$-poset $\E(n)$ it satisfies the following constraints. Here and from now on we denote by $z_0=I(g_0)=I(\emptyset )=1$ the invariant of the empty graph which is the graph with $r$ vertices and no edges. 
\begin{theorem}\label{the:rajoitteet}
The $G$-poset $\E(r)$ imposes on the $G$-poset $\E(n)$, $ n \geq r$, the constraints
\begin{eqnarray} 
z \in \mathbf{Z}^{m+1} \ \mathrm{s.t.} \ z \geq 0 \\ 
(E^T)^{-1}Dz \geq 0 \\
z_iz_j-\sum_k c_{ij}^k z_k = 0, \forall \ i,j \ \mathrm{s.t.} \ cv(g_i)+cv(g_j)\leq r
\end{eqnarray}
where the coefficients $c_{ij}^k$ come from the product formula, $cv(g_i)$ is
the number of vertices of the graph $g_i$ in connection with its edges, $E$ is
the $E$-transform of the $G$-poset $\E(r)$ (containing the empty graph) and $D$ is the following diagonal matrix:
\begin{equation}
D=diag\left( {n \choose r},{n-cv(g_1) \choose r-cv(g_1)},\ldots,{n-cv(g_{m}) \choose r-cv(g_m)}\right).
\end{equation}
\end{theorem}
We postpone the proof to the next subsections. We will call a vector $z$
\emph{r-graphic} if all the constraints above are 
satisfied. This is a weak notion of vector $z$ being graphic. There is not
necessarily a graph with the parameters $z$. However even if $z$ is not fully
graphic, the constraints provide useful insights into large $G$-posets. 

The constraints in Theorem \ref{the:rajoitteet} have several significant
consequences. Firstly they allow us to find a triangular system of lower and
upper bounds for $k=2,3,\ldots,m$ s.t. 
\begin{equation}
L(z_1,z_2,\ldots,z_{k-1})\leq z_k \leq U(z_1,z_2,\ldots,z_{k-1}),
\end{equation}
where $z$ is assumed to be in the large $G$-poset $\E(n)$. We will show both
linear and nonlinear bounds. Notice that $0 \leq z_1 \leq {n \choose 2}$ is
an optimal bound for $z_1$ since we always choose $g_1=a_{12}$. The triangular system makes looping over all $r$-graphic vectors $z$ very easy. This will be discussed in Sections \ref{sec:loop} and \ref{sec:larger}.

\subsection{Linear Inequalities}
To tie different $G$-posets together we need the following lemma.
\begin{lemma}\label{lem:suht}
Let $g$ be a graph with $cv(g)\leq r$. Then
\begin{equation}
\sum_{v(a)=r,a \subseteq h} I(g)(a)={n-cv(g) \choose r-cv(g)}I(g)(h),
\end{equation}
where the sum is over all $r$-vertex subgraphs $a$ of the graph $h$ and $n$ is the number of vertices in the graph $h$.
\end{lemma}

\begin{proof}
Fix the labels of the graph $g$ i.e. consider one single monomial of the
invariant $I(g)$. It is a simple matter to confirm that the number of $r$-sets containing the fixed graph $g$ is ${n-cv(g) \choose r-cv(g)}$.
\end{proof}
This lemma can be utilized as follows. Let $x$ denote the vector
$[1,I(g_1),I(g_2)$, $\ldots,$ $ I(g_m)]$ consisting of all invariants of the $G$-poset $\E(r)$. Next let 
\begin{equation}\label{diagonaali}
D=diag\left( {n \choose r},{n-cv(g_1) \choose r-cv(g_1)},\ldots,{n-cv(g_m) \choose r-cv(g_m)}\right).
\end{equation}
Consider now any linear inequality $c^Tx \leq 0$ inside the $G$-poset $\E(r)$,
i.e. for all graphs $g$ in the $G$-poset $\E(r)$ it holds that $c^Tx \leq 0$
when $x$ is evaluated at $g$. According to Lemma \ref{lem:suht} by summing
over all $r$-vertex-subsets of the invariants in the larger $G$-poset $\E(n)$
we have the relation $c^TDx \leq 0$.
\begin{example}
In the $G$-poset $\E(4)$ we find for instance that
\begin{equation}
[-1,2,0,-1,0,0,0,0,0,0,0]x\leq0,
\end{equation}
where the order of graphs is determined by Example \ref{ex:aaa}. Recall the $E$-transform is
\begin{displaymath}
E=\left( \begin{array}{ccccccccccc}
1 & 0 & 0 & 0 & 0 & 0 & 0 & 0 & 0 & 0 & 0 \\
1 & 1 & 0 & 0 & 0 & 0 & 0 & 0 & 0 & 0 & 0 \\
1 & 2 & 1 & 0 & 0 & 0 & 0 & 0 & 0 & 0 & 0 \\
1 & 2 & 0 & 1 & 0 & 0 & 0 & 0 & 0 & 0 & 0 \\
1 & 3 & 0 & 3 & 1 & 0 & 0 & 0 & 0 & 0 & 0 \\
1 & 3 & 1 & 2 & 0 & 1 & 0 & 0 & 0 & 0 & 0 \\
1 & 3 & 0 & 3 & 0 & 0 & 1 & 0 & 0 & 0 & 0\\ 
1 & 4 & 1 & 5 & 1 & 2 & 1 & 1 & 0 & 0 & 0 \\
1 & 4 & 2 & 4 & 0 & 4 & 0 & 0 & 1 & 0 & 0 \\
1 & 5 & 2 & 8 & 2 & 6 & 2 & 4 & 1 & 1 & 0\\ 
1 & 6 & 3 & 12 & 4 & 12 & 4 & 12 & 3 & 6 & 1 \\
\end{array} \right).
\end{displaymath}
The inequality $-1+2I(g_1)-I(g_3)\leq 0$ implies the inequality
\begin{equation}
-{n \choose 4}+2{n-2 \choose 2}I(g_1)-I(g_3) \leq 0.
\end{equation}
in the $G$-poset $\E(n)$ since $cv(\emptyset )=0$, $cv(g_1)=2$ and $cv(g_3)=4$. 
\end{example}

We can express this neatly by using the $E$-transform.
\begin{proposition}\label{pro:kat}
Let $E$ be the  $E$-transform of the $G$-poset $\E(r)$. Let $D$ be the diagonal
matrix (\ref{diagonaali}) above. The $G$-poset $\E(r)$ imposes the following linear constraints on the $G$-poset $\E(n)$ when $r\leq n$.
\begin{equation}
\forall c\in \mathbb{R}^{m+1} \ \mathrm{s.t.}\ Ec\leq0\ : \ c^TDz\leq0.
\end{equation}
\end{proposition}
But we can further simplify this result.
\begin{proposition}\label{the:lie}
Let $E$ and $D$ be the matrices as in Proposition \ref{pro:kat}. Let $z$ be the vector of graph invariants in the $G$-poset $\E(n)$ as described above. Then all the linear constraints given by Proposition \ref{pro:kat} are satisfied if the inequality
\begin{equation}
(E^{-1})^TDz \geq 0
\end{equation}
holds.
\end{proposition}
\begin{proof}
The set of vectors $c \in \mathbb{R}^{m+1}$ satisfying $Ec\leq0$ is a convex simplex. We find that all such vectors $c$ are generated by the formula $-E^{-1}y$, where $y\geq0$. It is therefore sufficient to confirm that $-y^T(E^{-1})^TDz \leq 0$ for all $y\geq0$. This follows if $-I(E^{-1})^TDz \leq 0$, where $I$ is the identity matrix.
\end{proof}
This proves the linear constraints in Theorem \ref{the:rajoitteet}.

\subsection{Nonlinear Equalities and Inequalities}
As explained in \cite{Tomi3}, the product formula for $I(g_i)I(g_j)$ in the $G$-poset $\E(r)$ is general if $cv(g_i)+cv(g_j) \leq r$. Naturally these apply also to $\E(n)$.
\begin{proposition}
Invariants of the $G$-poset $\E(r)$, $N=|\E(r)|$ satisfy 
\begin{equation}
I(g_i)I(g_j)=\sum_{k=1}^N \left(\sum_{h=1}^N (-1)^{e_{k1}-e_{h1}}e_{kh}e_{hi}e_{hj}\right)I(g_k),
\end{equation}
in the $G$-poset $\E(n)$, where $e_{ij}$ is the $ij^{\mathrm{th}}$ entry in the $E$-transform of the $G$-poset $\E(r)$ and $cv(g_i)+cv(g_j) \leq r$.
\end{proposition}
This proves the nonlinear part of Theorem \ref{the:rajoitteet}.

It is possible to incorporate the nonlinear constraints with the linear
constraints to produce strong lower and upper bounds for any graph
invariants. Let $g_1=a_{12}$, $g_2=a_{12}a_{34}$  and $g_3=a_{12}a_{13}$. We
develop a nonlinear lower bound $L_3$ and an upper bound $U_3$, for the invariant $I(g_3)$ s.t.
\begin{equation}
L_3(I(g_1)) \leq I(g_3) \leq U_3(I(g_1)).
\end{equation}

First we find by linear programming the best possible lower and upper bounds for our invariant $I(g_3)$ in terms of constant $1$, invariants $I(g_1)$ and $I(g_1)^2$ in $\E(5)$ (we had to choose $\E(5)$ instead of $\E(4)$ because $I(g_3)$ cancels in the inequality obtained in $\E(4)$). In the example we have maximized the surface area under the lower bound and minimized the surface area under the upper bound. We find that in $\E(5)$ 
\begin{equation}
-I(g_1)+\frac{2}{5}I(g_1)^2 \leq I(g_3) \leq \frac{1}{2}I(g_1)+\frac{1}{4}I(g_1)^2,
\end{equation}
see Figure \ref{fig:bnd}. 
\begin{figure}[!htk]
\begin{center}
\includegraphics[width=7cm,height=5cm]{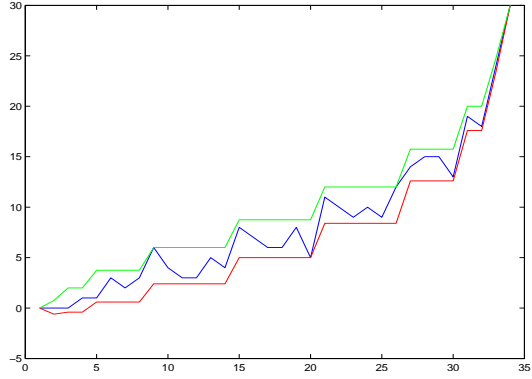}
\end{center}
\vspace{-0.5cm}
\caption{Lower and upper bounds for $I(g_3)$ in $\E(5)$.}\label{fig:bnd}
\end{figure}
Since Lemma \ref{lem:suht} applies only to linear invariants we substitute $I(g_1)^2=I(g_1)+2I(g_2)+2I(g_3)$ obtained from the product formula in $\E(4)$. We get
\begin{equation}
-3I(g_1)+4I(g_2) \leq I(g_3) \leq \frac{3}{2}I(g_1)+I(g_2).
\end{equation} 
Next we generalize this inequality in the $G$-poset $\E(n)$ by multiplying all the invariants $I(g_i)$ by ${n-cv(g_i) \choose 5-cv(g_i)}$. Thus we obtain
\begin{equation}
-2{n-2 \choose 3}I(g_1)+4(n-4)I(g_2) \leq {n-3 \choose 2}I(g_3) \leq \frac{3}{2}{n-2 \choose 3}I(g_1)+(n-4)I(g_2)
\end{equation}
since $cv(g_1)=2,cv(g_2)=4$ and $cv(g_3)=3$. Next we use again the relation $I(g_2)=\left(I(g_1)^2-I(g_1)-2I(g_3)\right)/2$ and obtain after simplifications the following proposition.
\begin{proposition}\label{pro:raja3}
Let $I(g_1)=I(a_{12})$ and $I(g_3)=I(a_{12}a_{13})$. Then in the $G$-poset $\E(n)$, where $n \geq 5$ we have
\begin{eqnarray}
I(g_3) \geq \frac{-\left(3{n-2 \choose 3}+2(n-4)\right)I(g_1)+2(n-4)I(g_1)^2}{{n-3 \choose 2}+4(n-4)} \\
I(g_3) \leq \frac{\frac{n-4}{2}I(g_1)^2+\left(\frac{3}{2}{n-2 \choose 3}-\frac{n-4}{2} \right)I(g_1)}{{n-3 \choose 2}+n-4}.
\end{eqnarray}
\end{proposition}
See Figure \ref{fig:bnd_e7}. 
\begin{figure}[!htk]
\begin{center}
\includegraphics[width=7cm,height=5cm]{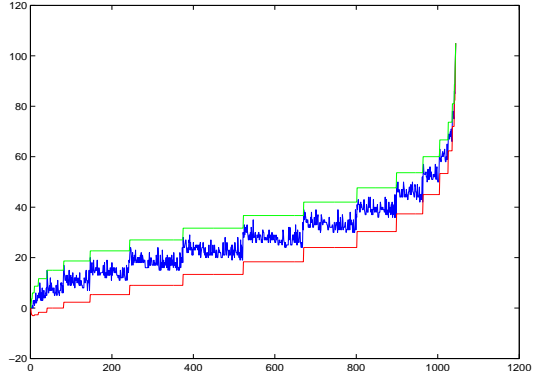}
\end{center}
\vspace{-0.5cm}
\caption{Lower and upper bounds for $I(g_3)$ in $\E(7)$ given by Propostion \ref{pro:raja3}.}\label{fig:bnd_e7}
\end{figure}

\subsection{Looping Over Weakly Graphic Invariants}\label{sec:loop}
Suppose we want to loop through all the vectors $z$, $z_i=I(g_i)(h)$, where $g_i\in\E(r)$ and $h\in \E(n)$, $n \geq r$ such that they are $r$-graphic i.e. all the constraints imposed by the $G$-poset $\E(r)$ are satisfied.

It seems to be more natural to loop over the variables $z_i$ in reversed order from $z_N$ to $z_1$. The constraint $(E^{-1})^TDz \geq 0$ is readily in a triangular form i.e. starting from the last row gives the constraints 
\begin{eqnarray}
{n-cv(g_N) \choose r-cv(g_N)}z_N \geq 0 \\ \nonumber
{n-cv(g_{N-1}) \choose r-cv(g_{N-1})}z_{N-1}+(-1)^{e_{N,1}-e_{N-1,1}}e_{N-1,N}{n-cv(g_N) \choose r-cv(g_N)}z_N \geq 0 \\ \nonumber
\ldots
\end{eqnarray}
We have used the facts that $e_{i,i}=1$ and on the diagonal the exponents of $-1$ are even.

Thus we have a triangular system of lower bounds and they are by themselves
already sufficient to guarantee that $z$ satisfies the linear constraints for
$r$-graphic vectors. However it is not clear from this triangular system when to stop 
adding the variable $z_N$. Although the polyhedron $(E^{-1})^TDz \geq 0,z \geq
0$ contains only a finite amount of integer points, we do not see that some $z_N$ is too large until we have tried all possibilities for $z_{N-1},z_{N-2},\ldots,z_1$.

A better way to deal this problem is to find also a triangular system of upper bounds.

Let $\hat{z}=(E^{-1})^Tz$ be the \emph{orthogonal parameters} in $\E(r)$. Only
one of the components of $\hat{z}$ is one and the rest of them are zero if $z$ is evaluated in a graph belonging to $\E(r)$. Thus if we sum over all $r$-vertex subsets of some graph in $\E(n)$ we get $\hat{z}=(E^{-1})^TDz$ and these satisfy
\begin{equation}
\sum_{g_k \in \E(r)} \hat{z}_k = {n \choose r}
\end{equation}
yielding
\begin{equation}\label{eq:rajo}
\hat{z}_i \leq {n \choose r}-\sum_{j = i+1}^N \hat{z}_j.
\end{equation}

Thus we are able to write a triangular system of upper bounds for $z$ as follows. Notice that 
\begin{equation}\label{eq:orto}
\hat{z_i}=\sum_{k=i}^N (-1)^{e_{i1}-e_{k1}}{n-cv(g_k) \choose r-cv(g_k)}e_{ki}z_k.
\end{equation}
Once we know the values of $z_N,z_{N-1},\ldots,z_{i+1}$ and $z_0=1$, by (\ref{eq:rajo}) we have
\begin{eqnarray}
\hat{z}_i \leq {n \choose r}-\sum_{k=i+1}^N \hat{z}_k \\ \nonumber
={n \choose r}-\sum_{k=i+1}^N\sum_{h=k}^N (-1)^{e_{k1}-e_{h1}}{n-cv(g_h) \choose r-cv(g_h)}e_{hk}z_h.
\end{eqnarray}

Next expand $\hat{z}_i$ by equation (\ref{eq:orto}) 
\begin{eqnarray}
\sum_{l=i}^N (-1)^{e_{i1}-e_{l1}}{n-cv(g_l) \choose r-cv(g_l)}e_{li}z_l \\ \nonumber
\leq{n \choose r}-\sum_{k=i+1}^N\sum_{h=k}^N (-1)^{e_{k1}-e_{h1}}{n-cv(g_h) \choose r-cv(g_h)}e_{hk}z_h \\ \nonumber
\Leftrightarrow \\ \nonumber
{n-cv(g_i) \choose r-cv(g_i)}z_i \leq {n \choose r}-\sum_{k=i+1}^N\sum_{h=k}^N (-1)^{e_{k1}-e_{h1}}{n-cv(g_h) \choose r-cv(g_h)}e_{hk}z_h \\ \nonumber
-\sum_{l=i+1}^N (-1)^{e_{i1}-e_{l1}}{n-cv(g_l) \choose r-cv(g_l)}e_{li}z_l \\ \nonumber
\Leftrightarrow \\ \nonumber
{n-cv(g_i) \choose r-cv(g_i)}z_i \leq {n \choose r} \\ \nonumber
- \sum_{h=i+1}^N {n-cv(g_h) \choose r-cv(g_h)} \left((-1)^{e_{i1}-e_{h1}}e_{hi}+\sum_{k=i+1}^h  (-1)^{e_{k1}-e_{h1}}e_{hk} \right) z_h \\ \nonumber
\Leftrightarrow \\ \nonumber
{n-cv(g_i) \choose r-cv(g_i)}z_i \leq {n \choose r} - \sum_{h=i+1}^N {n-cv(g_h) \choose r-cv(g_h)}\sum_{k=i}^h (-1)^{e_{k1}-e_{h1}}e_{hk}z_h
\end{eqnarray}

In matrix notation the lower bounds and upper bounds can be stated as follows.
\begin{proposition}
The triangular lower bound is 
\begin{equation}
Lz \leq Dz,
\end{equation}
where $L=((E^{-1})^T-I)D$. Similarly the upper bound in matrix notation
becomes 
\begin{equation}
Dz \leq {n \choose r}-Uz,
\end{equation}
where $U=T(E^{-1})^T D$ and $T$ is the upper triangular matrix with ones:
\begin{displaymath}
T=
\left( \begin{array}{ccccc}
0 & 1 & 1 & \cdots & 1\\
0 & 0 & 1 & \cdots & 1\\
\vdots & & & & \\
0 & 0 & 0 & \cdots & 1\\
0 & 0 & 0 & \cdots & 0
\end{array} \right).
\end{displaymath}
\end{proposition}

Sometimes there are reasons to loop variables in order $z_1,z_2,\ldots,z_N$. This can be achieved at least by linear programming. However we are not able to confirm that the system is equally tight as $(E^{-1})^TDz \geq 0,z \geq 0$.

Let $L_3(z_0,z_1,n)$ and $U_3(z_0,z_1,n)$ be the bounds for $I(g_3)$ given by Proposition \ref{pro:raja3}. Let us denote by $[LDz]_i$ the $i^{\mathrm{th}}$ element of the vector $\lceil LDz \rceil$ and $[UDz]_i$ the $i^{\mathrm{th}}$ element of the vector $\lfloor LDz \rfloor$.

Clearly $z_1$, the number of edges, satisfies $0\leq z_1 \leq {n \choose 2}$ in $\E(n)$. Secondly the bounds for $z_3$ are given by Proposition \ref{pro:raja3}. From $z_1$ and $z_3$ we can solve the $z_2=(z_1^2-z_1-2z_3)/2$ thus giving a good start for our loop.

In practice for the rest of the variables the nonlinear bounds obtained in the previous section become very complicated and thus we restrict to linear bounds in this section. However we describe a method in Section \ref{sec:larger} to incorporate also the nonlinear constraints in looping.

\begin{example}
In $\E(4)$ the lower bound found by linear programming is 
\begin{small}
\begin{displaymath}
L=\left [\begin {array}{ccccccccccc} 
0&0&0&0&0&0&0&0&0&0&0\\\noalign{\medskip}
0&0&0&0&0&0&0&0&0&0&0\\\noalign{\medskip}
0&0&0&0&0&0&0&0&0&0&0\\\noalign{\medskip}
-3&2&-1&0&0&0&0&0&0&0&0\\\noalign{\medskip}
0&-2/3&0&2/3&0&0&0&0&0&0&0\\\noalign{\medskip}
0&-1&2&1/2&3/2&0&0&0&0&0&0\\\noalign{\medskip}
0&0&0&-1/3&1&1/3&0&0&0&0&0\\\noalign{\medskip}
-4/5&4/5&-4/5&-1&7/5&4/5&7/5&0&0&0&0\\\noalign{\medskip}
0&0&-1/3&0&-1/2&1/3&0&1/6&0&0&0\\\noalign{\medskip}
0&-1/3&2/3&2/3&-1&-1&-1&7/6&4/3&0&0\\\noalign{\medskip}0
&0&0&0&0&0&0&-1/12&0&1/3&0\end {array}\right ]
\end{displaymath}
\end{small}
and the upper bound is
\begin{small}
\begin{displaymath}
U=\left [\begin {array}{ccccccccccc} 
0&0&0&0&0&0&0&0&0&0&0\\\noalign{\medskip}
6&0&0&0&0&0&0&0&0&0&0\\\noalign{\medskip}
0&1/2&0&0&0&0&0&0&0&0&0\\\noalign{\medskip}
0&1&2&0&0&0&0&0&0&0&0\\\noalign{\medskip}
0&0&0&1/3&0&0&0&0&0&0&0\\\noalign{\medskip}
{\frac {12}{11}}&-{\frac {12}{11}}&{\frac {16}{11}}&{\frac {12}{11}}&0&0&0&0&0&0&0\\\noalign{\medskip}
0&1/3&-2/3&0&1&0&0&0&0&0&0\\\noalign{\medskip}
0&2/3&0&-4/3&2&2/3&2&0&0&0&0\\\noalign{\medskip}
0&0&0&1/4&-3/4&0&-3/4&1/2&0&0&0\\\noalign{\medskip}
0&0&0&1/6&-1/2&-1/3&-1/2&5/6&2/3&0&0\\\noalign{\medskip}
0&0&0&0&0&1/12&0&-1/6&-1/3&1/2&0\end {array}\right ].
\end{displaymath}
\end{small}
\end{example}

The matrices $L$ and $U$ give optimized lower and upper bounds in the $G$-poset
$\E(r)$. To generalize these bounds in $\E(n)$, according to Lemma
\ref{lem:suht}, we multiply the matrices with the diagonal matrix $D$.

\section{Strongly Graphic Values}
Suppose we have the integer vector $z_i,i=1\ldots r$. When can we say that $z$
is \emph{graphic} i.e. there exists a graph $g$ s.t. $\forall
i:I(g_i)(g)=z_i$ for a fixed sequence of graphs $g_1,\ldots,g_r$? 

It is clear that the vector $z$ is graphic if $(E^{-1})^Tz \in
\{0,e_1,\ldots,e_{r}\}$, where $e_i$ are the elementary unit vectors. This is
because the rows of the $E$-transform are graphic vectors inside $\E$. The
resulting unit vector $e_j$ has the index $j$ of the corresponding graph
$g_j$. However if we aks which vectors $z$ are graphic corresponding to some
graph outside $\E$, the question is much more difficult. 

If we restrict to the trivial permutation group $G=1_G$, there exists many
more or less simple ways to characterize the graphic vectors. In particular
certain 'parity checks' can confirm this \cite{Tomi4}.

Could it be that the product formulas inside some moderate size $G$-poset imply
sufficient constraints for the small graph invariants?

We found negative answer to this. To be more precise, we list in Table \ref{taulu7} the distribution of values of $I(a_{12}a_{34})$ for
all positive integer vectors $z$ satisfying $z_1=|z|=7$ and
\begin{itemize}
\item[i] All multiplication constraints s.t. $|g_i|+|g_j|\leq 4$.
\item[ii] The correct distribution of $I(a_{12}a_{34})$ for graphs $G$ with
  $|G|=7$.
\end{itemize}
We used the $G$-poset $\E(8,4)$ to carry out the multiplications. Thus the
products $I(a_{12}a_{34})^2$, $I(a_{12}a_{34})$$I(a_{12}a_{13})$ and
$I(a_{12}a_{13})^2$ for instance are covered. The distributions are given by
enumerator polynomials s.t. the term $c_ix^i$ tells the number of vectors having $z_2=i$ is $c_i$.

It can be seen that the product constraints in the $G$-poset $\E(\infty,2d)$
are still insufficient to show that the invariants of degree less or equal to
$d$ are graphic. There are (two) vectors $z$ having $z_2=2$ corresponding to
the graph $a_{12}a_{34}$ which is impossible together with $z_1=7$.
\begin{table}[!htk]
\begin{small}
\caption{Distribution of positive integer solutions with constraints and the
  correct distribution.}\label{taulu7}
\begin{center}
\begin{tabular}{|l|l|}
\hline   
i: & $1 + 2x^2+ 8x^3+ 66x^4+ 322x^5+ 1245x^6+ 4029x^7+ 10748x^8+ 21092x^9$\\
     & $+ 28967x^{10}+ 28292x^{11}+ 18989x^{12}+ 8771x^{13}+ 3068x^{14}$\\
     & $+ 851x^{15}+ 203x^{16}+ 49x^{17}+ 10x^{18}+ 2x^{19}+ x^{20}+x^{21}$\\
ii: & $1 + x^4+ x^5+ 4x^6+ 4x^7+ 7x^8+ 11x^9+ 14x^{10}+ 18x^{11}+ 23x^{12}$\\
     & $+ 22x^{13}+ 21x^{14}+ 20x^{15}+ 13x^{16}+ 8x^{17}+ 5x^{18}+ 2x^{19}$\\
     & $+ x^{20}+ x^{21}$\\
\hline
\end{tabular}
\end{center}
\end{small}
\end{table}

To handle the general case with the permutation group $S_n$, we need to
develop a theory of local $G$-posets.

\subsection{Local Invariants}\label{sec:localinvariants}
It was shown above that the internal multiplication laws in $\E(\infty,d)$, 
i.e. the products $I(g)I(h)$ such that $|g|+|h| \leq d$, are not able to
characterize graphic vectors $z$ if $|z|>d$. 

Let $A \dunion B$ be the graph formed by a disjoint union of the graphs $A$
and $B$. Since $I(g)(A \dunion B)=I(g)(A)+I(g)(B)$ for connected invariants $I(g)$, we
may restrict ourselves to the case of graphic values of the connected graphs.

In Theorem $5$ of \cite{Tomi3}, we saw how $c=(E^{-1})^Tz$ implies the
existence of a suitable graph $\coprod c_i g_i$ if all $c_i \geq 0$. We state this in
the following proposition.
\begin{proposition}\label{gpro}
A sufficient condition for $z \in \mathbb{Z}_+^r$ to be graphic is
\begin{equation}\label{con223}
(E^{-1})^T z \geq 0,
\end{equation}
where $E$ is the $E$-transform of the connected graphs $g_1,\ldots,g_r$ in question.
\end{proposition}
\begin{proof}
The graph $\coprod c_i g_i$ contains subgraphs according to $E^T c=z$.
\end{proof}

\begin{example}
Sorted eigenvalues of the adjacency matrix $A$ are in 1-1 correspondence with the
coefficients of the characteristic polynomial.

These are clearly graph invariants since any permutation (or unitary)
transform $U^{-1}AU$ preserves the characteristic polynomial $det(A-zI)$.

The characteristic polynomial is
\begin{equation}
\sum_{\rho \in S_n} (-1)^\rho (a_{1 \rho(1)}-\delta(\rho(1)-1)z) \cdots (a_{n \rho(n)}-\delta(\rho(n)-n)z)
\end{equation}
which boils down to 
\begin{equation}
\sum_{\mathbf{p} \in \mathcal{P}(n)} \sum_{i=0}^{n_1} (-z)^{i+\mathbf{p}}
I(C_{p_1-i,p_2,p_3,\ldots,p_n}),
\end{equation}
where $C_{n_1,n_2,\ldots,n_r}$ is the graph containing $n_i$ disjoint cycles of lenght
$i$ and $\mathbf{p} \in \mathcal{P}(n)$ means looping over partitions of $n$.

The $G$-poset of these disjoint cycle graphs is easy. We get the
$E$-transform by
\begin{equation}
I(C_{n_1,\ldots,n_r})(C_{m_1,\ldots,m_r}) = \prod_i {m_i \choose n_i}.
\end{equation}
Thus by multiplying this $E$-transform with appropriate coefficients we obtain
the graphic values for the coefficients of the characteristic polynomials. We
leave more explicit characterization, analogous to Proposition \ref{gpro}, as
a research problem. 
\end{example}

The condition (\ref{con223}), however, is not \emph{necessary} but it can be
extended by breaking the group $G$ into smaller parts. 

The subgroups which we consider are the stabilizers $Stab_{S_n}(S)$ of
sequences of vertices $S$. As the stabilizer subgroups are also permutation
groups, all the results earlier in this paper apply for the
\emph{local invariants} defined below.

Let 
\begin{equation}
I_S(g)=\sum_{\rho \in \mathrm{Stab}_{S_n}(S)/\mathrm{Stab}(g)}   \rho(g),
\end{equation}
where $S$ is a sequence of vertices in $g$, $g$ is a monomial representing a graph and
$\mathrm{Stab}_{S_n}(S)$ is a point-wise stabilizer of the sequence $S$. Thus the
sum is over all permutations of $g$ which fix the vertices in $S$. We will call
the indices in $S$ \emph{fixed points}.

Notice that by using a permutation $\pi_{S,T}$ mapping $S \rightarrow T$, we are
able to evaluate the invariant in different locations of the target graph.

We denote by $I_S(g)^{\pi_{S,T}}$ the invariant $I_S(g)$ evaluated in fixed
vertices $T$ in $h$. Also we denote by $g(S)$ the partially labeled graph $g$
having the fixed points $S$.

For instance 
\begin{equation}
I_1(a_{12}a_{13})^{\pi_{1,4}}( a_{45}a_{46}a_{47})= I_1(a_{12}a_{13})( a_{45}a_{46}a_{47}^{\pi_{1,4}^{-1}} )=3. 
\end{equation}

Let $g$ be a labeled graph $(V,E)$, where $V=\{1,2,\ldots,n\}$. Let $V_k$ be
the set of ordered subsets of $V$ of cardinality $k$. We define an equivalence
relation $\equiv_g$ on $V_k$ as follows. For $S,T \in V_k$ $S \equiv_g T$ if
there exists $\rho \in \mathrm{Stab}_{S_n}(g)$ such that
$g(S)^\rho=g(T)$. 

\begin{lemma}\label{palautus}
Let $g$ be a graph in $\E(n)$ and let $V=\{1,2,\ldots,n\}$ be the set
of vertices. We can restore the global invariant $I(g)$ in two ways:
\begin{itemize}
\item[i]
We have
\begin{equation}
\sum_{T \in V_{|S|}} I_S(g)^{\pi_{S,T}} = |Orb_{\mathrm{Stab}(g)}(S)| I(g),
\end{equation}
where $Stab(g)=\{ \pi \in S_{|S|} | g^\pi=g\}$ and $S_{|S|}$ is the symmetric permutation group permuting the fixed points.
\item[ii]
We have
\begin{equation}
\sum_{S \in V_{|S|}/\equiv_g} I_S(g)=I(g),
\end{equation}
where the equivalence relation $\equiv_g$ is defined above.
\end{itemize}
\end{lemma}
\begin{proof}
To prove the first part we write the sum as
\begin{equation}
\sum_{\pi \in S_n/S_{n-|S|}} \sum_{\rho \in S_n/\mathrm{Stab}(g,S)} a^{\rho \pi},
\end{equation}
where $a$ is some monomial with unit coefficient in $R[a_{ij}]$ representing the graph $g$ and the
fixed points $S$ are determined by this particular labeling.

In the latter sum the coefficient of each monomial is clearly one. Thus the
coefficient in the total sum of the monomial, say $a$, is the number of
choices of $\pi$ s.t. $a^\pi=a$. Thus the coefficient equals
\begin{equation}
|\{\pi \in S_n/S_{n-|S|} | a^\pi=a \}|.
\end{equation}
Notice this is independent with respect to different isomorphic choices of
$a$.

Since the set $S_n/S_{n-|S|}$ corresponds to the set $Orb_{S_n}(S)$, we have 
that $\{S^\pi | \pi \in S_n/S_{n-|S|}\}$ equals $Orb_{S_n}(S)$ and therefore
\begin{equation}
|\{\pi \in S_n/S_{n-|S|}|a^\pi=a\}|=|Orb_{\mathrm{Stab}(g)}(S)|.
\end{equation}
The second part is a sum over cosets of $\equiv_g$ where each coset consists
of all monomials in $I_S(g)$.
\end{proof}

\begin{figure}[!htk]
\begin{center}
\includegraphics[width=3cm,height=1cm]{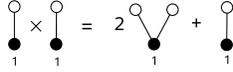}
\end{center}
\vspace{-0.5cm}
\caption{Product of local invariants.}\label{fig:1sepa}
\end{figure}

\begin{example}
Figure \ref{fig:1sepa} shows how the normally connected invariants like
$I(a_{12}a_{13})$ behave like unconnected graphs when the group is broken by stabilizing
the vertex $1$. We can solve 
\begin{equation}
I_1(a_{12}a_{13})=\left( I_1(a_{12})^2-I_1(a_{12}) \right)/2
\end{equation}
for simple graphs.
\end{example}

This happens in general and can be stated as follows.

Let us call a (partially labeled) graph $g(S)$ such that $g(S) \neq g(S)\cap S$, $S$-\emph{connected}
($S$-\emph{unconnected} in the opposite case) if for each pair of vertices
$i,j \in g(S)$, where at least the other, say $i \notin S$, there is a path
connecting $i$ and $j$ without traveling through the vertices in $S$. 

For $A \subseteq S$ let $\chi_A$ denote the set of $A$-connected graphs which are not $A\cup
\{i\}$-connected for any $i \in S \setminus A$.

\begin{proposition}\label{pro:lgene}
All local invariants with the fixed points $S$ are generated/separated by
the $A$-connected invariants, where $A \subseteq S$, together with the internal
edges $a_{ij}$, where $i,j \in S$.
\end{proposition}
\begin{proof}
It is clear that the edges $a_{ij}$ inside $S$ generate/separate the graph
$g(S) \cap S$. Let $g(S) \setminus_e S$ denote the graph obtained by removing
the internal edges.

To the rest of the graph it suffices to write $g(S) \setminus_e S$ into the representation 
\begin{equation}
g(S) \setminus_e S \leftrightarrow \sum_{A \subseteq S} \sum_{h \in \chi_A } n_h^A h(A),
\end{equation}
where $n_h^A$ is the number of components $h(A)$ in $g(S)\setminus_e
S$. Theorem $5$ in \cite{Tomi3} now applies with the permutation group
$Stab_G(S)$ and implies that these graphs which are separated, are also
generated by the $S$-connected invariants. Finally the whole invariant is
recovered by $I_S(g)=I_S(g \setminus_e S)\prod_{i,j \in S} a_{ij}$.
\end{proof}

\begin{example}
Let $S=\{1,2,3\}$ and $g(S)=a_{14}a_{15}a_{16}a_{26}a_{37}a_{38}a_{78}$, where
the interior has already been removed for simplicity. The representation is
\begin{displaymath}
g(S) \leftrightarrow (2 a_{14})_1 + (a_{14}a_{24})_{1,2} + (a_{34}a_{35}a_{45})_3,
\end{displaymath}
where we have used several times the same labels for the vertices outside $S$
and the subscripts indicate the subset $A$ of fixed vertices.
\end{example}

The next lemma is needed in generalization of Proposition \ref{gpro}. Let
$g\setminus_e S$ be the graph $g$ with the internal edges in $S$
removed. We do not remove the edges where only one of the end points belongs
to $S$. Let $g\setminus S$ denote the graph with all the vertices in $S$
removed together with the edges in connection to $S$. We also denote by $A
\coprod_S B$ the graph formed by disjoint union of graphs $A$ and $B$ s.t. the
vertices in $S$ coincide. For instance $a_{12}a_{23} \coprod_1 a_{13} = a_{12}a_{23}a_{14}$.

\begin{lemma}\label{lem:summ}
An $S$-connected invariant of form $g=h \setminus_e S$, where $g \neq \emptyset$ satisfies
\begin{equation}
I_S(g)(A \coprod_S B)=I_S(g)(A) + I_S(g)(B).
\end{equation}
\end{lemma}
\begin{proof}
Consider first the case where $g$ contains only one vertex $v$ outside $S$. Then
$g$ is of form $a_{s_1,v}a_{s_2,v} \cdots a_{s_l,v}$, $s_j \in S$. Let $a(b)$ denote the
evaluation of the monomial $a$ in $b$. Then the equation is
satisfied since 
\begin{eqnarray}
\sum_{\rho \in S_n/\mathrm{Stab}(g,S)}a^\rho(A \coprod_S B) &=& \sum_{v \notin S}
a_{s_1,v}a_{s_2,v}\cdots a_{s_l,v}(A \coprod_S B) \\ \nonumber
&=& \sum_{v \in A\setminus S} a_{s_1,v}a_{s_2,v}\cdots a_{s_l,v}(A) +\sum_{v \in
  B\setminus S}
a_{s_1,v}a_{s_2,v}\cdots a_{s_l,v}(B).
\end{eqnarray}

If $g$ contains at least two vertices $i,j \notin S$, then by definition of
$S$-connectedness we have that the distance $d(i,j)<\infty$, where the
distance $d(i,j)$ is the length of the shortest path between $i$ and $j$ in
$g\setminus S$.

Suppose $A \coprod_S B$ had a subgraph $g'(S)$ isomorphic to $g(S)$ having two
vertices $i'$ and $j'$ s.t. $i' \in A\setminus S$ and $j'\in B\setminus S$. Then the distance
$d(i',j')=\infty$ since the only connections must go through the vertices in
$S$. Thus $A \coprod_S B$ contains no subgraphs isomorphic to $g(S)$ contained
in both components $A$ and $B$. Thus we may sum them separately and the result follows.
\end{proof}

Let us call $g(S) \cap S$ \emph{the interior} of the local graph $g(S)$.
Let $\E_S=\{g_1(S)$, $g_2(S)$, $\ldots$, $g_r(S)\}$ be an $G$-poset of $S$-connected
local invariants such that $g_i(S) \setminus_e S \neq \emptyset$ with the same interior i.e. $g_i(S) \cap S = g_j(S)\cap S$. We denote by $z(S)$ the vector of local invariants with the fixed
points $S$:
\begin{equation}
z(S)=[I_S(g_1),I_S(g_2),\ldots,I_S(g_r)]^T.
\end{equation}

The following result generalizes Proposition \ref{gpro}.
\begin{proposition}\label{pro:nesu}
The local vector $z(S) \in \mathbb{Z}_+^r$ is graphic if
\begin{equation}
(E_S^{-1})^T z(S) \in \mathbb{Z}_+^r.
\end{equation}
\end{proposition}
\begin{proof}
Write a graph $h$ in form $h=\coprod_S n_ig_i$, where $n_i\in \mathbb{Z}_+$
are the multiplicities of the connected components in $\E_S$ of $h$. According
to Lemma \ref{lem:summ} we have
\begin{equation}
I_S(g_j)(h)=\sum n_i e_{ij},
\end{equation}
and we have $z(S)=E^T[n_1,\ldots,n_r]^T$ and $[n_1,\ldots,n_r]^T =
(E^T)^{-1}z(S)$. Thus whenever we can find non-negative multiplicities $n_1,\ldots,n_r$, there
is the corresponding graph $\coprod_S n_ig_i$.
\end{proof}

Let the local $G$-poset $\E_S$ be decomposed as $\E_S=\mathcal{C}_S \cup
\mathcal{U}_S$, where $\mathcal{C}_S=\{g_1(S),\ldots,g_r(S)\}$ is the local $G$-poset of $S$-connected
graphs in $\E_S$ with the same interior and $\mathcal{U}_S$ is the $G$-poset of
$S$-unconnected graphs. Let $C_S$ be the $E$-transform of $\mathcal{C}_S$.

Notice that the $S$-connected invariants parametrize all the graphs in $\E_S$
together with the internal edges. Thus once we have the parameters
$I_S(g_1),\ldots,I_S(g_r)$ we can solve the
multiplicities
\begin{equation}
m = (C_S^{-1})^T [I_S(g_1),\ldots,I_S(g_r)]^T
\end{equation}
which give the representation $\coprod_S m_ig_i$ for all graphs in $\E_S$
with the same interior as the connected graphs $g_i(S)$ have.

\begin{proposition}
In the local $G$-poset the $S$-unconnected invariants with the same interior
can be evaluated by the following recursion:
\begin{eqnarray}
I(\coprod_S n_i g_i)(\coprod_S m_i g_i) = \\ \nonumber
\sum_{g^* \leq \coprod_S n_i g_i,g^*   \subseteq G^* } I(\coprod_S n_i g_i
\setminus g^*)(\coprod_S m_i g_i \setminus G^*)I(g^*)(G^*),
\end{eqnarray}
where $G^*$ is any $S$-connected component in $\coprod_S m_i g_i$ with the positive
coefficient $m_{G^*}$ and the sum is over all $g^* \subseteq G^*$ such that $g^*=\coprod_S
p_i g_i$ satisfies $\forall i:p_i \leq n_i$.
\end{proposition}
\begin{proof}
This recursion is analogous to the identity ${n \choose k} = {n-1 \choose
  k}+{n -1 \choose k-1}$. 

The idea is to divide the graph $\coprod_S m_i g_i$ in two parts, the graph
$\coprod_S m_i g_i - G^*$ and the graph $G^*$. Then consider separately how many
times $\coprod_S n_i g_i$ is contained in both of these and the result
follows. 
\end{proof}

By this recursion we are able to reduce the problem of evaluating 
\begin{equation}
I(\coprod_S n_ig_i)(\coprod_i m_ig_i)
\end{equation}
to the problem of determining $I(\coprod_S n_ig_i)(g_k)$, where $g_k$ is $S$-connected. 

\subsection{Very Restricted Case}\label{sec:lok}
Before characterizing the local invariants in different locations we show
a restricted class of graphs which allow reconstruction of the local
parameters from the global vector $z$. 

In the following $\E_1$ denotes the $G$-poset of $1$-connected local graphs.
\begin{definition}\label{def:res}
The graph $g$ is
$\E_1$-\emph{restricted} iff 
\begin{itemize}
\item[i] for every vertex $v\in g$ there is one maximal connected local graph $h \in \E_1$ satisfying
  $I_1(h)^{\pi_{1v}}(g)=1$ and
\item[ii] all the other local invariants in $\E_1$ at $v$ can be obtained by
  $I_1(k)^{\pi_{1v}}(g)=I_1(k)(h)$.
\end{itemize}
\end{definition}

This condition is needed to avoid the situation where two local graphs
$h,k$ satisfy $I_1(h)^{\pi_{1v}}(g)=1$ and $I_1(k)^{\pi_{1v}}(g)=1$, whereas 
$I_1(k)(h)=0$ and $I_1(h)(k)=0$. This is crucial in the
proposition below to compute the local parameters from $z$.

Let $z(i)$ be the local invariant vector at vertex $i$, that is 
\begin{equation}
z(i)=[I_1(g_1)^{\pi_{1i}},\ldots,I_1(g_r)^{\pi_{1i}}].
\end{equation}
For each $z(i)$ we associate a $E_1^T$-\emph{transform pair}
$\hat{z}(i)=(E_1^T)^{-1} z(i)$. The parameters $\hat{z}(i)$ describe the
maximal local graph in the neighbourhood of the vertex $i$ and
$\hat{z}(i)\in\{0,e_1,\ldots,e_r\}$, where $e_i$ are the elementary unit
vectors describing the type of the neighbourhood. 

\begin{proposition}
Let $z$ be the global invariant vector $z=[I(g_1),\ldots,I(g_r)]$ and the local
$G$-poset of graphs $\E_1=\{g_1(1),\ldots,g_r(1)\}$. If $h$ is
$\E_1$-restricted graph, then
\begin{equation}
\sum_i \hat{z}(i)=(E_1^{-1})^TD_1 z,
\end{equation}
where $D_1=diag(|Orb_{Stab(g_1)}(1)|,\ldots,|Orb_{Stab(g_r)}(1)|)$ and $E_1$
is the $E$-transform of the local $G$-poset $\E_1$.
\end{proposition}
\begin{proof}
Let $y=\sum_i \hat{z}(i)$. Let $g_1,g_2,\ldots g_r$ be in the increasing order
by the partial order of subgraph containment. For all
subgraphs $g_r$ in $h$, there are $|Orb_{Stab(g_r)}(1)|$ vertices of type
$g_r(1)$. Since all vertices have their neighborhoods in $\E_1$, we get
$y_r = |Orb_{Stab(g_r)}(1)| z_r$.

Next for each subgraph $g_{r-1}$ in $h$, there are
$|Orb_{Stab(g_{r-1})}(1)|-I_1(g_{r-1})(g_{r})y_r$ vertices
of type $g_{r-1}(1)$ since we have to subtract the vertices of type $g_r(1)$
multiplied by the number of local graphs $g_{r-1}(1)$ in $g_r(1)$.

In general
\begin{equation}
y_k = |Orb_{Stab(g_{k})}(1)|z_k - \sum_{i>k} e_{ik}^1 y_i, 
\end{equation}
where $e_{ij}^1$ is the $ij^{\mathrm{th}}$ element in $E_1$. This implies
$D_1z = E_1^Ty$ from which we solve $y = (E_1^{-1})^TD_1z$.
\end{proof}

\begin{lemma}\label{lem:lakkk}
Let $h$ be $\E_1$-restricted and $y = (E_1^{-1})^TD_1 z$. We can compute the number of connected vertices
in $h$ by
\begin{equation}
n=\sum_i y_i.
\end{equation}
\end{lemma}
\begin{proof}
Each $\hat{z}(i) =e_{k_i} \neq 0$ corresponds to a vertex with the neighborhood specified by $e_{k_i}$.
\end{proof}

Notice it does not matter which connected vertices in $g_i$ we choose to be
the fixed point $1$. As long as the hypothesis that every vertex in $h$
belongs to one of the types defined by $\E_1=\{g_1(1),\ldots,g_r(1)\}$ holds, the result applies.

We continue the splitting of $S_n$ further.

For simplicity we select $\E_1$ so that all local graphs connected to the
vertex $1$ up to the distance $r$ are in $\E_1$. By distance we mean the
length of the shortest path between vertices. We also include the connections
between all vertices up to distance $r$. To make $\E_1$ finite we may
restrict the vertex degree to $d$. This restricts also the set of graphs that 
can be described by these local parameters.

Let $\E$ the infinite $G$-poset of graphs with vertex degree at most $d$ and 
denote by $U_r(i,g)$ the local subgraph of $g$ at the fixed point $i$ up to
the distance $r$ from $i$. We use also notation $U_r(j,\hat{z}(i)), j \neq i$
to denote the local graph in $\E_{1,2}$ such that the distance of any vertex
to the fixed point $j$ is limited to $r$ in the local graph defined by $\hat{z}(i)$.

Choose
\begin{equation}
\E_1=\{U_r(1,g^\pi) | \pi \in S_n, g \in \E \}
\end{equation}
which is in other words all the neighborhoods up to distance $r$ appearing in
$\E$ with the fixed point $1$. We have added the permutation $\pi \in S_n$ to
make sure all the vertex types will be included.

We continue the process by choosing
\begin{equation}\label{eq:process}
\E_{1,2,\ldots,k} = \{ U_r(1,g^\pi) \cap U_r(2,g^\pi) \cap \cdots \cap
U_r(k,g^\pi) | \pi \in S_n,g \in \E \},
\end{equation}
where $1,\ldots,k$ are all fixed points.

\begin{example}\label{ex:lok}
In Figure \ref{fig:lok1} we have limited to $U_1$-neighborhoods of trivalent
graphs. In Figures \ref{fig:lok12} and \ref{fig:lok123} we have followed the
construction given above. For $\E_{1,2}$ and $\E_{1,2,3}$ we have only drawn
one copy of all labellings of the fixed points.

\begin{figure}[!htk]
\begin{center}
\includegraphics[width=8cm,height=1.5cm]{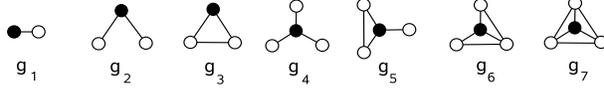}
\end{center}
\vspace{-0.5cm}
\caption{Local $G$-poset $\E_1$.}\label{fig:lok1}
\end{figure}

\begin{figure}[!htk]
\begin{center}
\includegraphics[width=4.3cm,height=1.7cm]{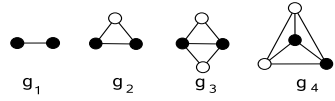}
\end{center}
\vspace{-0.5cm}
\caption{Local $G$-poset $\E_{1,2}$.}\label{fig:lok12}
\end{figure}

\begin{figure}[!htk]
\begin{center}
\includegraphics[width=1.9cm,height=1.2cm]{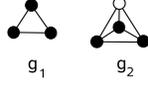}
\end{center}
\vspace{-0.5cm}
\caption{Local $G$-poset $\E_{1,2,3}$.}\label{fig:lok123}
\end{figure}
\end{example}

Let $P_{1,\ldots,i}$ be a projector from $\E_{1,\ldots,i-1}$ to
$\E_{1,\ldots,i}$ in the process (\ref{eq:process}). More precisely, let $r_{i-1}=|\E_{1,\ldots,i-1}|$ and
$r_i=|\E_{1,\ldots,i}|$. Then $P_{1,\ldots,i}$ maps from $\mathbb{Z}^{r_{i-1}}
\rightarrow \mathbb{Z}^{r_i}$.

\begin{proposition}\label{pro:hukl}
If $h(k_1,\ldots,k_{i-1})$ is
$\E_{1,\ldots,i}=\{\emptyset,g_1(1,\ldots,i),\ldots,g_r(1,\ldots,i)\}$-restricted
and $z(k_1,\ldots,k_{i-1})$ is the local graph invariant vector at $k_1,\ldots,k_{i-1}$ given by $z_t=I_{k_1,\ldots,k_{i-1}}(g_t(k_1,\ldots,k_{i-1}))(h(k_1,\ldots,k_{i-1}))$, then
\begin{equation}\label{eq:ekl}
\sum_{j \neq k_1,\ldots,k_{i-1}} \hat{z}(k_1,\ldots,k_{i-1},j)=(E_{1,\ldots,i}^{-1})^TD_{1,\ldots,i}P_{1,\ldots,i} z(k_1,\ldots,k_{i-1}),
\end{equation}
where $D_{1,\ldots,i}=diag(|Orb_{Stab(g_1(1,\ldots,i-1))}(i)|,\ldots,|Orb_{Stab(g_r(1,\ldots,i-1))}(i)|)$,
$E_{1,\ldots,i}$ is the $E$-transform of the local $G$-poset
$\E_{1,\ldots,i}$ and $P_{1,\ldots,i}$ is the projector defined above. 
\end{proposition}
\begin{proof}
Same proof as above except now there are $|Orb_{Stab(g_k(1,\ldots,i-1))}(i)|$
sequences of type $g_k(1,\ldots,i)$ for each occurrence of $z_k(1,\ldots,i-1)$ subgraph not contained in larger subgraphs.
\end{proof}

\begin{example}
We choose the local $G$-posets $\E_1$ and $\E_{1,2}$ in Example
\ref{ex:lok}. The $E$-transform of $\E_1$ is 
\begin{displaymath}
E_1=\left[ \begin{array}{lllllll}
1 & 0 & 0 & 0 & 0 & 0 & 0\\
2 & 1 & 0 & 0 & 0 & 0 & 0\\
2 & 1 & 1 & 0 & 0 & 0 & 0\\
3 & 3 & 0 & 1 & 0 & 0 & 0\\
3 & 3 & 1 & 1 & 1 & 0 & 0\\
3 & 3 & 2 & 1 & 2 & 1 & 0\\
3 & 3 & 3 & 1 & 3 & 3 & 1 
\end{array} \right].
\end{displaymath}
The inverse is 
\begin{displaymath}
E_1^{-1}=\left[ \begin{array}{lllllll}
1 & 0 & 0 & 0 & 0 & 0 & 0\\
-2 & 1 & 0 & 0 & 0 & 0 & 0\\
0 & -1 & 1 & 0 & 0 & 0 & 0\\
3 & -3 & 0 & 1 & 0 & 0 & 0\\
0 & 1 & -1 & -1 & 1 & 0 & 0\\
0 & 0 & 0 & 1 & -2 & 1 & 0\\
0 & 0 & 0 & -1 & 3 & -3 & 1 
\end{array} \right].
\end{displaymath}
The $G$-poset $\E_1$ is not complete thus explaining why the inverse formula
in \cite{Tomi3} does not apply.
Let $z=[9,14,1,4,2,0,0]^T$. We have $D_1=\mathrm{diag}([2,1,3,1,1,2,4])$ and
\begin{equation}
\sum_{i} \hat{z}(i)=(E_1^{-1})^TD_1z = [2,1,1,2,2,0,0]^T.
\end{equation}
By choosing arbitrarily $\hat{z}(1)=e_1$,$\hat{z}(2)=e_1$,
$\hat{z}(3)=e_2$, $\hat{z}(4)=e_3$, $\hat{z}(5)=e_4$, $\hat{z}(6)=e_4$,
$\hat{z}(7)=e_5$, $\hat{z}(8)=e_5$ we can solve the $E_1^T$-transform pairs 
$z(1)=E_1^T\hat{z}(1)=e_1$,$z(2)=e_1$, $z(3)=2e_1+e_2$, $z(4)=2e_1+e_2+e_3$,
$z(5)=3e_1+3e_2+e_4$, $z(6)=3e_1+3e_2+e_4$,$z(7)=3e_1+3e_2+e_3+e_4+e_5$,
$z(8)=3e_1+3e_2+e_3+e_4+e_5$. These in turn allow us to compute the local
neighborhoods with two fixed points. For instance
\begin{equation}
\sum_{j\neq 8} \hat{z}(8,j)=(E_{1,2}^{-1})^TD_{1,2}P_{1,2}z(8) = e_1+2e_2,
\end{equation}
where
\begin{displaymath}
E_{1,2}=\left[ \begin{array}{lll}
1 & 0 & 0\\
1 & 1 & 0\\
1 & 2 & 1\end{array} \right],
\end{displaymath}
$D_{1,2}=\mathrm{diag}([1,2,1])$ and
\begin{displaymath}
P_{1,2}=\left[ \begin{array}{lllllll}
1 & 0 & 0 & 0 & 0 & 0 & 0\\
0 & 0 & 1 & 0 & 0 & 0 & 0\\
0 & 0 & 0 & 0 & 0 & 1 & 0
\end{array} \right].
\end{displaymath}
\end{example}

\begin{example}\label{ex:lakko}
Let $g_1(1,2)=a_{12} \in \E_{1,2}$ and $g_1(1,2,3)=a_{12} \in
\E_{1,2,3}$. Then the equation (\ref{eq:ekl}) in the case $i=3$ multiplied by $E_{1,2,3}^T$ on
both sides states for instance that
\begin{equation}
\sum_{k \neq i,j} z_1(i,j,k) = |Orb_{Stab(g_1(1,2))}(3)|z_1(i,j) = (n-2)z_1(i,j).
\end{equation}
Notice we need to know what is the number of vertices $n$ in order to
calculate $D_{1,2}$ or $D_{1,2,3}$. We get $n$ by Lemma \ref{lem:lakkk}.
\end{example}

\subsection{Conjugate Equations}
In $\E_{1,\ldots,i}$ we say two elements $g(1,\ldots,i),h(1,\ldots,i)$ are
$\pi$-\emph{conjugates} iff they are isomorphic with respect to $Stab(1,\ldots,i)$
after permutation $\pi \in S_{\{1,\ldots,i\}}$:
\begin{equation}
g(1,\ldots,i)^{\pi} \cong_{Stab(1,\ldots,i)} h(1,\ldots,i)
\end{equation}
in this order.

For instance $a_{13}$ is $(12)$-conjugate to $a_{23}$ and $a_{24}$ in $\E_{1,2}$.

With invariant vectors conjugation is easy. For instance, let 
\begin{equation}
\E_{1,2}=\{g_1(12)=a_{13},g_2(12)=a_{23},g_3(12)=a_{13}a_{12},g_4(12)=a_{12}a_{23}\}.
\end{equation}

Then $z(1,2)=[z_1(1,2),z_2(1,2),z_3(1,2),z_4(1,2)]$ is $(12)$-conjugate to
$x(1,2)=[z_2(1,2),z_1(1,2),z_4(1,2),z_3(1,2)]$. We denote this by $z(1,2)^\pi = x(1,2)$.

Let $\hat{z}(i_1,i_2,\ldots,i_k) \in \{0,e_1,\ldots,e_{r_k}\}$ denote the local graph types with $k$ fixed points
with their $E_{1,\ldots,k}^T$-transform pairs
$z(i_1,i_2,\ldots,i_k)=E_{1,\ldots,k}^T\hat{z}(i_1,i_2,\ldots,i_k)$ such that 
\begin{equation}\label{eq:km1k}
\sum_{i_k=1,i_k \neq i_1,\ldots,i_{k-1}}^n \hat{z}(i_1,i_2,\ldots,i_k)=(E_{1,\ldots,k}^{-1})^TD_{1,\ldots,k}P_{1,\ldots,k}z(i_1,\ldots,i_{k-1}).
\end{equation} 

These equations imply an algorithm for reconstructing the graph $g$ from
$z$ where $z=[I(g_1)(g),\ldots,I(g_{r_1})(g)]$. First calculate
$x=(E_1^{-1})^TD_1z$ and then assign in arbitrary order 
the vertices $\hat{z}(1),\ldots,\hat{z}(n)$ such that $\sum_i
\hat{z}(i)=y$. Notice $n=\sum_i x_i$. Next calculate the parameters $z(i)=E_1^T\hat{z}(i)$ and solve
the systems 
\begin{equation}
\sum_{j \neq i} \hat{z}(i,j)=(E_{1,2}^{-1})^TD_{1,2}P_{1,2}z(i)
\end{equation}
and
\begin{equation}
z(i,j)^{(ij)} = z(j,i),
\end{equation}
where $z(i,j)=E_{1,2}^T \hat{z}(i,j)$. 

If all the conditions of the following theorem are satisfied by the parameters
$z(i)$, $i=1,2,\ldots,n$, we obtain the adjacency matrix $G$ of the graph $g$
via $z_1(i,j)$ assuming the first graph in $\E_{1,2}$ is $a_{12}$.

It is not sufficient that each pair of vertices satisfy the conjugation
equations.

Let us say that the sequence $\E_1,\E_{1,2},\ldots,\E_{1,2,\ldots,s}$ is \emph{reconstructible} if each $g_i(1,2,\ldots,k) \in \E_{1,2,\ldots,k}$ is uniquely determined by the graphs $U_r(g_i(1,2,\ldots,k),j), j \neq 1,2,\ldots,k$ in $\E_{1,2,\ldots,k+1}$. The sequence of $G$-posets $\E_1,\E_{1,2},\ldots,\E_{1,2,3,4}$ in the example \ref{ex:lok} is reconstructible.

Let $s$ smallest integer such that $\E_{1,2,\ldots,s+1}= \emptyset$ in the process
(\ref{eq:process}). In Example \ref{ex:lok} for instance $s=4$, since $\E_{1,\ldots,4}=\{a_{12}a_{13}a_{14}a_{23}a_{24}a_{34}\}$ and $\E_{1,\ldots,5}=\emptyset$. 
\begin{theorem}\label{the:graafisuus1}
Let the sequence $\E_1,\E_{1,2},\ldots,\E_{1,2,\ldots,s}$ be reconstructible. The local parameters $\hat{z}(i) \in \{0,e_1,\ldots,e_{r_1}\}$, $i=1,\ldots,n$ are graphic iff the local parameters with $k=2,3,\ldots,s$ fixed points obtained via
  (\ref{eq:ekl}) together with the process (\ref{eq:process}), satisfy the conjugate equations 
\begin{equation}\label{eq:konjugaa}
z(i_1,\ldots,i_k)^{\pi}=z(i_{\pi(1)},\ldots,i_{\pi(k)})
\end{equation}
for all permutations in $\mathfrak{S}_k$. 
\end{theorem}

\begin{proof}
Notice first that the parameters $z(1),\ldots,z(n)$ are graphic if the parameters $z(i_1,i_2)$ as a solution to the conjugate equations and the equations (\ref{eq:ekl}), are graphic. This follows from the reconstruction; $z(i_1)$ is fully determined by $z(i_1,i_2)$, $i_2 \neq i_1$ and if the parameters $z(i_1,i_2)$ describe an actual graph it confirms that the neighbourhoods of $z(1),\ldots,z(n)$ match against each other. Reconstruction guarantees that there is only one way to compile $z(i_1)$ out of $z(i_1,i_2),i_2 \neq i_1$.

We begin by showing that the system described by $z(i_1,i_2,\ldots,i_s)$ as a solution to conjugate equations and (\ref{eq:ekl}) is graphic. In $\E_{1,\ldots,s}$ the graphs contain only fixed points. (Suppose on the contrary that there is one vertex $v$ in the graph $g(1,\ldots,s) \in \E_{1,\ldots,s}$ which is not a fixed point. Then one can form intersection $U_r(g(1,\ldots,s),v)$ which is non-empty and therefore $\E_{1,\ldots,s+1}\neq \emptyset$ which is contradiction.)

To show that the system $z(i_1,\ldots,i_s)$ is graphic we have to show that whenever there is an edge $(i_j,i_k)$ in $z(i_1,\ldots,i_s)$, the same edge appears also in all the other $z(j_1,\ldots,j_s)$ neighbourhoods containing the end points $i_j,i_k \in \{j_1,\ldots,j_s\}$. This can be seen by finding suitable permutation $\pi$ s.t. $z(i_j,i_k,i_1,i_2,\ldots,i_s)=z(i_1,\ldots,i_s)^\pi$ and then noticing that all the subgraph parameters of $z(i_j,i_k)$ obtained by
\begin{equation}
\sum_{i_1 \neq i_j,i_k} \hat{z}(i_j,i_k,i_1) = (E_{1,2,3}^{-1})^T D_{1,2,3} P_{1,2,3} z(i_j,i_k)
\end{equation}
and furthermore by
\begin{equation}
\sum_{i_t \neq i_j,i_k,i_1,\ldots,i_{t-1}} \hat{z}(i_j,i_k,i_1,\ldots,i_{t-1}) = (E_{1,\ldots,t+2}^{-1})^T D_{1,\ldots,t+2} P_{1,\ldots,t+2} z(i_j,i_k)
\end{equation}
contain the edge $(i_j,i_k)$ if $z(i_j,i_k)$ contains it. This follows from the construction of $E_{1,\ldots,k}$-trasform matrices. 

Now by finding similarly permutation $\phi$ s.t. $z(i_j,i_k,j_1,\ldots,j_s)=z(j_1,\ldots,j_s)^\phi$ we find that the edge $(i_j,i_k)$ in $z(j_1,\ldots,j_s)$ is consitent with respect to $z(i_j,i_k)$ and therefore with respect to $z(i_1,\ldots,i_s)$. Thus the parameters $z(i_1,\ldots,i_s)$ are graphic.

Since the sequence $\E_1,\E_{1,2},\ldots,\E_{1,\ldots,s}$ is reconstructible we can read from $z(i_1,\ldots,i_s)$ the parameters $z(i_1,\ldots,i_{s-1})$ with one exception: if the graph $z(i_1,\ldots,i_{s-1})$ contains only fixed points, then it will not be separated from the empty graph. But if $z(i_1,\ldots,i_{s-1})$ represents a graph containing only the fixed points, we can deduce as above that the system is graphic one the conjugate equations and (\ref{eq:ekl}) are satisfied.

We continue the same process to the levels $z(i_1,\ldots,i_{s-2})$, $z(i_1,\ldots,i_{s-3})$, $\ldots$ $z(i_1,i_2)$. Thus we have that the system $z(i_1,i_2),\ldots,z(i_1,\ldots,i_s)$ is graphic and by reconstruction it represents the graph corresponding to the parameters $z(1),\ldots,z(n)$. Thus the parameters $z(1),\ldots,z(n)$ are graphic.
\end{proof}

We find that if the parameters $z(i,j)$ satisfy the conjugate equations, some graph can be reconstructed from the first components of $z(i,j)$ corresponding to $a_{12}\in \E_{1,2}$. This, however, does not imply that $z(1),z(2),\ldots,z(n)$ are graphic.

\begin{example}\label{ex:arra}
Consider for instance the parameters $\hat{z}(i),\hat{z}(i,j)$ using the
$G$-posets from Example \ref{ex:lok} and given in the following array, where the diagonal elements correspond to $z_{ii}=\hat{z}(i)$.
\begin{displaymath}
\mathcal{Z}=\left[ \begin{array}{llll}
e_3 & e_2 & 0 & e_2\\
e_2 & e_3 & e_2 & 0\\
0 & e_2 & e_3 & e_2\\
e_2 & 0 & e_2 & e_3 \end{array} \right].
\end{displaymath}
These parameters, indeed, satisfy all pairwise conjugate equations (all
elements in $\E_{12}$ happen to be self-conjugates in Example
\ref{ex:lok}) but this is not a graphic system of parameters since
$\hat{z}(1)$ being of type $e_3$ (see Figure \ref{fig:lok1}) and connected to
$2$ and $4$ should imply that $z_1(2,4)=1$ i.e. the vertices $2$ and $4$
should be connected. 

This is why we need all the parameters $z(i_1,i_2,i_3)$. The
equations (\ref{eq:ekl}) state that
\begin{equation}
\hat{z}(1,2,3)+\hat{z}(1,2,4)=(E_{1,2,3}^{-1})^TD_{1,2,3}P_{1,2,3} z(1,2)=e_1,
\end{equation}
which implies that if $\hat{z}(1,2,3)=e_1$, then $z(1,3,2)=e_1$ by conjugate
equation and we have a contradiction since $z(1,3)=0$ and cannot yield the
'child' $z(1,3,2)=e_1$ by equations (\ref{eq:ekl}).

On the other hand if $\hat{z}(1,2,4)=e_1$, then $z(2,4,1)=e_1$ and this
contradicts with $z_1(2,4)=0$. Thus the parameters in $\mathcal{Z}$ do not 
satisfy all the constraints.
\end{example}

We gave the result so that it is easy to generalize it to hyper graphs. 
It is obvious that we don't have to continue any further since all the parameters
$z(i_1,\ldots,i_{s+1})$ are zero. In fact, when working with hypergraphs, it
may be advantageous to use some sparse representation to the $z$-parameters
and introduce deeper terms $z(i_1,\ldots,i_k)$ only for those 'parents'
$z(i_1,\ldots,i_{k-1})$ which are non-zero. 

The result above implies the following result for the global invariant vector $z$.
However, the invariants in the vector $z=[I(g_1),\ldots,I(g_{r_1})]$ are
dictated by the local $G$-poset $\E_1$.

\begin{corollary}\label{cor:lgkd}
The graph invariant vector $z=[I(g_1),\ldots,I(g_{r_1})]$ corresponds to an
$\E_1$-restricted graph iff there are parameters $\hat{z}(i)\in
\{0,e_1,\ldots,e_{r_1}\}$ such that
\begin{equation}
\sum_{i} \hat{z}(i)=(E_1^{-1})^TD_1z
\end{equation}
and they satisfy the conditions of Theorem \ref{the:graafisuus1}.
\end{corollary}

Next consider the general case, where the graph is not restricted in any
way but the sets of local invariants in
$\E_1,\E_{1,2},\ldots,\E_{1,2,\ldots,k}$ are finite and the local $G$-posets
are constructed in the process (\ref{eq:process}).

The conditions in Theorem \ref{the:graafisuus1} imply the following
constraints for the $z(i)$, $z(i_1,i_2)$, $\ldots$, $z(i_1,\ldots,i_k)$-parameters.
\begin{corollary}\label{cor:koro}
The parameters $z(i),z(i_1,i_2),\ldots,z(i_1,\ldots,i_k)\geq 0 $ must satisfy 
\begin{itemize}
\item[i] the conjugate equations
\begin{equation}
z(i_1,\ldots,i_k)^{\pi}=z(i_{\pi(1)},\ldots,i_{\pi(k)})
\end{equation}
for all permutations in $S_k$, 
\item[ii] the sums 
\begin{equation}\label{lala12}
\sum_{i_k \neq i_1,\ldots,i_{k-1}}
z(i_1,\ldots,i_k)=D_{1,\ldots,k}P_{1,\ldots,k}z(i_1,\ldots,i_{k-1})
\end{equation}
and
\item[iii] the internal products of the local invariants which can be expressed as a linear
combination of the elements in the local $G$-poset
\begin{equation}
z(i_1,\ldots,i_k)_az(i_1,\ldots,i_k)_b =(E_{1,\ldots,k}^{-1} p)^T z(i_1,\ldots,i_k), 
\end{equation}
where 
\begin{displaymath}
p=[I_{1,\ldots,k}(g_a)(g_1)I_{1,\ldots,k}(g_b)(g_1),\ldots,I_{1,\ldots,k}(g_a)(g_{r_k})I_{1,\ldots,k}(g_b)(g_{r_k})]^T.
\end{displaymath}
\end{itemize}
\end{corollary}
\begin{proof}
Just extend the local $G$-poset $\E_1$ to the infinite $G$-poset $\E_1^*$ together
with the $E_1$-transform of infinite size. Do the same for $\E_{1,2},\ldots,
\E_{1,\ldots,k}$. Then the equations (\ref{eq:ekl}) imply (\ref{lala12}) when
the parameters in the finite parts $\E_1,\E_{1,2},\ldots,\E_{1,\ldots,k}$ are
considered.

The product formula must hold for the invariants for which the product terms
appear in $\E_{1,\ldots,k}$. This can be seen by writing local version of the product
formula in \cite{Tomi3} as follows. Suppose the product
$I_{1,\ldots,k}(g_a)I_{1,\ldots,k}(g_b)$ is the following linear combination:
\begin{equation}
I_{1,\ldots,k}(g_a)I_{1,\ldots,k}(g_b)=\sum_{g_l \in \E_{1,\ldots,k}} c_{ab}^l I_{1,\ldots,k}(g_l).
\end{equation}
Then the coefficients can be solved by 
\begin{equation}
c_{ab} = E_{1,\ldots,k}^{-1} p,
\end{equation}
where $p$ is the vector defined above. Naturally this product formula holds
only if all the terms appearing in the product are in $\E_{1,\ldots,k}$.
\end{proof}

Although the conjugate equations provide necessary and sufficient condition
for the graphic values in the restricted case, it is not easy to find whether
such a solution $z(i),z(i,j),\ldots,z(i_1,\ldots,i_s)$ exists. Even for
$z(i,j)$ this problem has been shown to be NP-complete if the number of
non-zero types is larger than $2$ \cite{Chrobak}. Interestingly, the complexity
is not known if the number of non-zero types is $2$.

We summarize this in the following corollary.
\begin{corollary}
Given the parameters $z=[I(g_1),\ldots,I(g_r)]$ and the partial solution
$z(i,j)=0$ for certain pairs $i,j$, the graph reconstruction is NP-complete in
the $\E_1$-restricted case if $|\E_1| >2$. 
\end{corollary} 
\begin{proof}
The only technical difference is due to the conjugate constraints which are not
present in \cite{Chrobak}. However by extending the problem we are able to
find a suitable reconstruction problem where the conjugate constraints do not
restrict finding the solution. Let the original problem be of finding 
\begin{displaymath}
Z=\left[ \begin{array}{llll}
\cdot & z_{12} & \cdots & z_{1n}\\
z_{21}& \cdot & \cdots & z_{2n}\\
\vdots & & \ddots & \\
z_{n1}& \cdots &z_{n, n-1} & \cdot \end{array} \right]
\end{displaymath}
such that the row sums equal to $\sigma(i)$ and the column sums equal to
$\theta(i)$. Now construct $Z'$ of size $2n \times 2n$ and of form 
\begin{displaymath}
Z'=\left[ \begin{array}{ll}
Z_{11}&Z_{12}\\
Z_{21}&Z_{22} \end{array} \right],
\end{displaymath}
where the row sums add up to $\sigma(i)$ on the first $n$ rows
and the column sums add up to $\theta(i)$ on the last $n$ columns. Also we
expect the first $n$ columns to add up to $\overline{\sigma}(i)$ and the last
$n$ rows add up to $\overline{\theta}(i)$, where $\overline{\sigma}$ denotes
the conjugation. We are expecting a solution of form
\begin{displaymath}
Z'=\left[ \begin{array}{ll}
0&Z\\
Z^H&0 \end{array} \right],
\end{displaymath}
where $Z^H$ denotes the Hermitian transpose of $Z$ i.e. $z_{ij}=\overline{z_{ji}}$.
 Clearly any solution $Z$ to the original problem is also solution to this
 problem with the conjugate constraints. 

We stated the corollary so that we can assert $z'_{ij}=0$ for $i,j \leq n$
and $z'_{ij}=0$ for $i,j \geq n+1$. The solution $Z'$ to this problem yields
clearly a solution $Z=Z_{12}$ to the original problem.

Any reconstruction of the adjacency matrix yields this solution in polynomial
time  since the local invariants are of finite degree and thus polynomial time
computable. Thus the problem is in NP and by the reasoning above NP-complete if
the partial solution is provided.
\end{proof}
\section{Graphs with Finitely Generated Local Neighbourhoods}
In this section we combine the results from the two previous sections to find a sufficient condition for the graphic tensor $z,z(i)$, $z(i_1,i_2)$
,$\ldots,z(i_1,\ldots,i_s)$. This result gives both sufficient and
necessary constraints for all graphs whose local neighbourhoods are generated
by the connected graphs in $\E_1$. This result is much more compact than
Corollary \ref{cor:lgkd} since the local $G$-posets can be infinitely large, yet parametrized by finitely many connected
components. For instance all graphs are locally finitely generated by the
locally connected graph invariant $I_1(a_{12})$. In this case also the graph
reconstruction problem has a polynomial time solution
\cite{Havel},\cite{Hakimi}. This corresponds to the neighbourhoods in Figure \ref{fig:adj}.

The parametrization is not complete in general but it is sufficient having
the constraint
\begin{equation}
(E_S^{-1})^Tz(S) \geq 0
\end{equation}
for the graphic vector $z(S)=[I_S(g_1),\ldots,I_S(g_r)]^T$. Theorem
\ref{the:graafisuus1} can be applied to the connected parameters to glue the
local parameters at different locations yielding the result.

We consider the local $G$-posets $\E_1,\E_{1,2},\ldots,\E_{1,\ldots,s}$ consisting of
$\{1,2,\ldots,k\}$-connected graphs for each $k=1,\ldots,s$ but
constructed otherwise similarly as before:
\begin{equation}\label{prosessi}
\E_k = \{U_r(1,g^\pi) \cap \cdots U_r(k,g^\pi) | \pi \in \mathfrak{S}_n, g \in \E \}
\setminus \ \mathcal{U}_{1,2,\ldots,k},
\end{equation}
where $\mathcal{U}_{1,2,\ldots,k}$ denotes the set of disconnected graphs minus
the interior graphs. We must add the interiors of all graphs in
$\E_{1,\ldots,k}$ since the interior is not generated by the graphs $g_i(S)
\setminus_e S \neq \emptyset$.

The projectors $P_k$ and the diagonal matrices $D_k$ are restricted on the
connected components but remain otherwise the same.
\begin{figure}[!htk]
\begin{center}
\includegraphics[width=5cm]{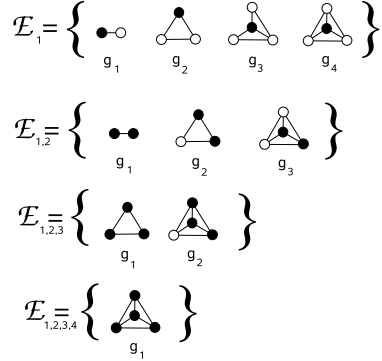}
\caption{Locally connected $G$-posets $\E_{1}$, $\E_{1,2}$, $\E_{1,2,3}$ and $\E_{1,2,3,4}$ with interiors.}\label{fig:con123}
\end{center}
\end{figure}
In Figure \ref{fig:con123} is shown certain locally connected $G$-posets which
are derived in the process \ref{prosessi}. Notice we are not anymore
restricted to trivalent graphs.

Let $\E_{1,\ldots,k}^*$ denote the $G$-poset of local graphs generated by the locally connected graphs in $\E_{1,\ldots,k}$.

We say the sequence $\E_1,\E_{1,2},\ldots,\E_{1,\ldots,s}$ of locally connected graphs is \emph{reconstructible} if the sequence $\E_1^*,\E_{1,2}^*,\ldots,\E_{1,\ldots,s}^*$ is reconstructible. The next lemma helps to characterize when the sequence of extended $G$-posets is graphic.

\begin{lemma}
The sequence $\E_1^*,\E_{1,2}^*,\ldots,\E_{1,\ldots,s}^*$ is reconstructible if the connected graphs in $\E_{1,\ldots,k}$ have linearly independent representation in terms of parameters in $\E_{1,\ldots,k+1}^*$ for all $k$.
\end{lemma}
\begin{proof}
Let $E_{1,\ldots,k}$ be the Mnukhin transform of $\E_{1,\ldots,k}^*$. If the parameters $z(i_1,\ldots,i_{k+1})$ obtained for locally connected graphs in $\E_{1,\ldots,k}$ by 
\begin{equation}
\sum_{i_{k+1} \neq i_1,\ldots,i_k}\hat{z}(i_1,\ldots,i_{k+1}) = (E_{1,\ldots,k+1}^{-1})^T D_{1,\ldots,k+1} P_{1,\ldots,k+1} z(i_1,\ldots,i_k)
\end{equation}
are linearly independent for each locally connected graph in $\E_{1,\ldots,k}$ then we can read the multiplicities $m_i$ of the graph 
\begin{displaymath}
m_1 g_1 \coprod_{1,\ldots,k} \cdots \coprod_{1,\ldots,k} m_rg_r
\end{displaymath}
from the parameters $z(i_1,\ldots,i_{k+1})$ since each connected graph in $\E_{1,\ldots,k}$ corresponds to a linearly independent vector.
\end{proof}

The sequence of connected graphs in Figure \ref{fig:con123} is reconstructible by this lemma.

\begin{theorem}
Connected vectors $z,z(i)$, $z(i_1,i_2)$, $\ldots,z(i_1,\ldots,i_s)$ are
graphic if
\begin{equation}\label{eq:ldld9}
\sum_{i_k \neq i_1,\ldots,i_{k-1}} z(i_1,\ldots,i_k)=D_kP_kz(i_1,\ldots,i_{k-1}),
\end{equation}
\begin{eqnarray}\label{eq:elele}
z(i_1,\ldots,i_k) \in \mathbb{Z}_+^{r_k} \\ \nonumber
(E_{1,\ldots,k}^{-1})^T z(i_1,\ldots,i_k) \geq 0
\end{eqnarray}
and the conjugate equations hold
\begin{equation}
\forall \pi \in \mathfrak{S}_k:z(i_1,\ldots,i_k)^{\pi}=z(i_{\pi(1)},\ldots,i_{\pi(k)})
\end{equation}
for all $k=1,\ldots,s$. Moreover this condition is both necessary and
sufficient for all graphs whose local neighbourhoods are generated/separated by the
connected graphs in $\E_1$.
\end{theorem}
\begin{proof}
The idea is that the connected graphs give a representation for $\E_{1,\ldots,k}^*$ and the equations (\ref{eq:ldld9}) and (\ref{eq:elele}) imply the required equations for the parameters in $\E_{1,\ldots,k}^*$ given by Theorem \ref{the:graafisuus1}.

First we show that the conjugate equations for the connected components imply also isomorphisms for the disconnected local graphs. In other words if
$z(i_1,\ldots,i_r)^\pi=z(i_{\pi},\ldots,i_{\pi(k)})$ holds for the connected
components, it must hold for the disconnected components too.

Let $z_u(S)$ denote some disconnected component at the local
neighbourhood of $S$. By Proposition \ref{pro:lgene} there exists a
polynomial $P_u$ s.t. $z_u(S)=P_u(z(S))$, where $z(S)$
is the connected invariant vector. We have $z(S)^\pi=z(S^\pi)$ by the
conjugate equations and on the other hand is is obvious that
$z_u(S^\pi)=P_u(z(S^\pi))$. This implies that $z_u(S^\pi)=P_u(z(S)^\pi)$. On
the other hand we have
\begin{equation}
I_S(g_i^\pi)I_S(g_j^\pi) = \sum_{g_k} c_{ij}^k I_S(g_k^\pi)
\end{equation}
implying that $P_u(z(S)^\pi)=P_u(z(S))^\pi$. Thus $z_u(S^\pi)=z_u(S)^\pi$.

Secondly all the $1,2,\ldots,k$-connected are also $1,2,\ldots,k-1$-connected. Thus the equations 
\begin{equation}a
\sum_{i_k \neq i_1,i_2,\ldots,i_{k-1}} I_{1,\ldots,k}(g)^{\pi_{(1,\ldots,k),(i_1,\ldots,i_k)}}(h) \\ \nonumber
=|Orb_{Stab(g,1,2,\ldots,k-1)}(k)|I_{1,\ldots,k-1}(g)^{\pi_{(1,\ldots,k-1),(i_1,\ldots,i_{k-1})}}(h)
\end{equation}a
implied by the proof of Proposition \ref{pro:hukl} yield the equations
\begin{equation}
\sum_{i_k \neq i_1,\ldots,i_{k-1}} z(i_1,\ldots,i_k)=D_kP_kz(i_1,\ldots,i_{k-1})
\end{equation}
determining the $1,\ldots,k$-connected parameters.
\end{proof}


\subsection{Products in Local $G$-posets}
We investigate the product formula for local invariants where the factors
have different fixed points $A$ and $B$.

Let $\equiv_a$ be the equivalence relation of ordered sets defined by $\Delta_1 \equiv_a \Delta_2$ if $a(\Delta_1)
\cong a(\Delta_2)$ with respect to the permutation group $\mathrm{Stab}(A)$,
where $A \subseteq \Delta_1,\Delta_2$. Then let $V_{a}^{ab}$ be the set of
ordered sets of vertices of size $|(A \cup B) \setminus A|$ in $V \setminus
A$, where $V=\{1,2,\ldots,n\}$ and $A,B \subseteq V$. Define $V_{b}^{ab}$
similarly by subtracting the set $B$.

Consider the local $G$-poset $\E_{A \cup B}=\{g_1,\ldots,g_r\}$ having the
$E$-transform $E_{A \cup B}$ and containing all the invariants appearing in
the following product. 
\begin{proposition}\label{sufff}
The product of two local invariants at vertices $A$ and $B$ equals
\begin{equation}
I_{A}(a) I_{B}(b) = \sum_{g \in \E_{A \cup B}} c_{a,b}^g I_{A \cup B}(g),
\end{equation}
where the coefficients $c_{a,b}^g \in \mathbb{Z}_+$ are given as a vector over $g$ by
\begin{equation}
\mathbf{c}_{a,b}=E_{A\cup B}^{-1} \left( \sum_{\Delta_a \in V_{a}^{ab}/\equiv_a}
  \mathbf{I}_{A \cup \Delta_a}^a \right)\left( \sum_{ \Delta_b \in V_{b}^{ab}/\equiv_b
  } \mathbf{I}_{B \cup \Delta_b}^b \right),
\end{equation}
where 
\begin{equation}
\mathbf{I}_{A \cup \Delta_a}^a=[I_{A \cup \Delta_a}(a)(g_1),\ldots,I_{A \cup \Delta_a}(a)(g_r) ]^T,
\end{equation}
and
\begin{equation}
\mathbf{I}_{B \cup \Delta_b}^b=[I_{B \cup \Delta_b}(b)(g_1),\ldots,I_{B \cup \Delta_b}(b)(g_r) ]^T.
\end{equation}
\end{proposition}
\begin{proof}
We write
\begin{equation}
I_{A}(a)=\sum_{\Delta_a \in V_{a}^{ab}/\equiv_a} I_{A \cup \Delta_a}(a)
\end{equation}
and $I_{B}(b)$ similarly. As the product is invariant only in $\mathrm{Stab}(A
\cup B)$ we must find the linear combination in $\E_{A \cup B}$. This can be
done by taking the inverse of the product of the two vector sums given above.
\end{proof}

\section{Ramsey Invariants}\label{sec:ram}
We want to express $I({\overline K}_k)$ in terms of basic graph
invariants. We start with the following lemma.

\begin{lemma}\label{lem:anti}
Let $g$ be the structure of the monomial $a_{\tau_1}a_{\tau_2}\cdots a_{\tau_d}$. Then 
\begin{eqnarray}
\sum_{\rho \in S_n/Stab(a_{\tau_1}a_{\tau_2}\cdots a_{\tau_d}) }(1-a_{\rho(\tau_1)})(1-a_{\rho(\tau_2)})\cdots (1-a_{\rho(\tau_d)}) \\ \nonumber
=\sum_{a \subseteq g}(-1)^{|a|}\frac{I(a)(g)|Stab(a)|}{|Stab(g)|}I(a), \label{eq:kaks}
\end{eqnarray}
where $|a|$ is the number of edges in $a$ and the sum is over all unlabeled subgraphs of the graph $g$. 
\end{lemma}
\begin{proof}
 The number of terms in the first sum is $n!/|Stab(g)|$. Each of these terms contains $I(a)(g)$ monomials of the invariant $I(a)$. Since the number of monomials in $I(a)$ is $n!/|Stab(a)|$ we get the coefficient $\frac{I(a)(g)|Stab(a)|}{|Stab(g)|}$ for $I(a)$. 
\end{proof}

By using this we have
\begin{eqnarray}
I({\overline K}_k)=\sum_{\rho \in S_n/Stab(a_{12}a_{13}\cdots a_{k-1 \ k})} (1-a_{\rho(12)})(1-a_{\rho(13)})\cdots (1-a_{\rho(k-1 \ k)}) \\ \nonumber
= \sum_{A \subseteq K_k}(-1)^{|A|}\frac{I(A)(K_k)|Stab(A)|}{|Stab(K_k)|}I(A).
\end{eqnarray}
This yields
\begin{eqnarray}
I({\overline K}_k)=\sum_{A \subseteq K_k}(-1)^{|A|}\frac{|Stab_{S_n}(A)|}{(n-k)!|Stab_{S_k}(A)|}I(A),
\end{eqnarray}
following from $I(A)(K_k)=k!/|Stab_{S_k}(A)|$ and $Stab(K_k)=(n-k)!k!$. Let $cv(A)$ denote the number of vertices connected to edges in the graph $A$. Since $|Stab_{S_n}(A)|=|Stab_{S_{|A|}}(A)|(n-cv(A))!$ and $|Stab_{S_k}(A)|=|Stab_{S_{|A|}}(A)|(k-cv(A))!$, we obtain the following result.

\begin{proposition}\label{pro:anti}
We have
\begin{equation}
I({\overline K}_k)=\sum_{A \subseteq K_k} (-1)^{|A|}{n-cv(A) \choose k-cv(A)} I(A)
\end{equation}
\end{proposition}
\begin{example}
$I({\overline K}_3)={n \choose 3}$-$(n-2)I(a_{12})+I(a_{12}a_{13})$$-I(a_{12}a_{13}a_{23})$.
\end{example}

\subsection{Ramsey Problem}
Let $f\in \mathbb{Z}^m$ be assigned according to
\begin{displaymath}
f_i=\left\{ \begin{array}{cc}
(-1)^{e(g_i)}{n-cv(g_i) \choose k-cv(g_i)} & \mathrm{if} \  cv(g_i)\leq k \ \mathrm{and} \ |g_i|\neq {n \choose 2} \\
1+(-1)^{e(g_i)}{n-cv(g_i) \choose k-cv(g_i)} & \mathrm{if} \ cv(g_i)=k \ \mathrm{and} \ |g_i|={n \choose 2}\\
0 & \ \mathrm{if} \ cv(g_i)>k. \\

\end{array}\right.
\end{displaymath}
Then the invariant vectors $z$ of the invariants of the $G$-poset $\E(r)$
satisfy
\begin{equation}
f^Tz=I(K_k)+I({\overline K_k})
\end{equation}
in the $G$-poset $\E(n)$ according to Proposition \ref{pro:anti}.

\begin{theorem}
An upper-bound for the Ramsey number $r(k)$ can be obtained with the help of a
medium size graph $G$-poset $\E(r)$, $k\leq r \leq n$, by finding the minimal $n$ s.t.
\begin{displaymath}
min \ f^Tz > 0,
\end{displaymath}
where $z\geq 0,z\in\mathbb{Z}$ and the minimization is carried out subject to the constraints
\begin{eqnarray}
-(E^{-1})^TDz \leq 0 \ \ \ (I)\\ \nonumber
z_iz_j-\sum_k c^k_{ij}z_k=0 \ \forall i,j\  \mathrm{s.t.}\ cv(g_i)+cv(g_j)\leq r \ \  (II)\\ \nonumber
\end{eqnarray}
where the coefficients $c_{ij}^k$ are calculated in the $G$-poset $\E(r)$,
\begin{equation}
D=diag\left( {n \choose r},{n-cv(g_1) \choose r-cv(g_1)},\ldots,{n-cv(g_m) \choose r-cv(g_m)}\right).
\end{equation}
and $E$ is the  $E$-transform of $\E(r)$.
\end{theorem}

We draw a couple of examples s.t. the $x$-axis is the value of the first graph
invariant $I(a_{12})$. The curves are calculated by finding the minimum by linear programming and sequential quadratic programming in Matlab.

\begin{figure}[!htk]
\begin{center}
\includegraphics[width=7cm,height=5cm]{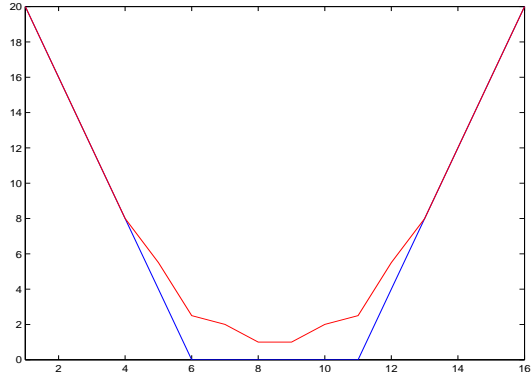}\label{fig:R3_4}
\end{center}
\vspace{-0.5cm}
\caption{Lower bounds for $I(K_3)+I(\overline{K_3})$ in $\E(6)$. These bound are calculated in $\E(4)$. The better bound is calculated subject to the constraints I and II. The worse is subject to the linear constraint I only. The nonlinear bound already shows that $r(3)\leq6$.}
\end{figure}
\begin{figure}[!htk]
\begin{center}
\includegraphics[width=7cm,height=5cm]{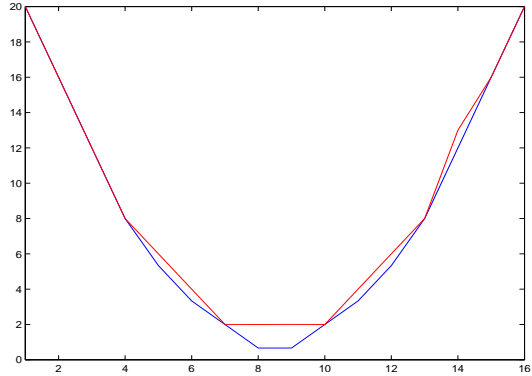}\label{fig:R3_5}
\end{center}
\vspace{-0.5cm}
\caption{Lower bounds for $I(K_3)+I(\overline{K_3})$ in $\E(6)$. These bound are calculated in $\E(5)$. The better bound is calculated subject to the constraints I and II. The worse is subject to the linear constraint I only. Now also the linear bound calculated in $\E(5)$ shows that $r(3)\leq6$. }
\end{figure}

We depict the lower bound for $I(K_4)+I(\overline{K_4})$ in $\E(5)$ in Figure \ref{fig:R4_5}.
\begin{figure}[!htk]
\begin{center}
\includegraphics[width=7cm,height=5cm]{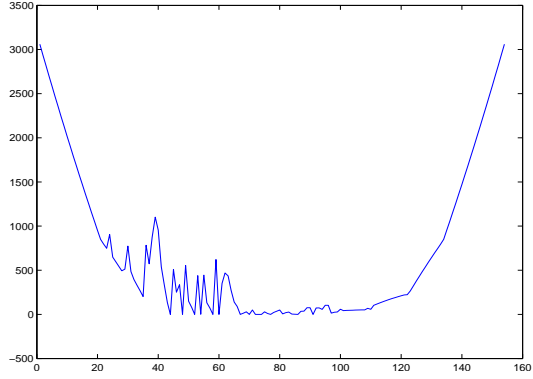}\label{fig:R4_5}
\end{center}
\vspace{-0.5cm}
\caption{The lower bound for $I(K_4)+I(\overline{K_4})$ in $\E(18)$. This bound is calculated in $\E(5)$ subject to the constraints I and II. The maximum number of iterations have been exceeded.}
\end{figure}
As the figure shows, numerical methods fail to confirm $r(4)=18$. To proceed we have to develop integer programming
techniques for restricting the solutions. 

\subsection{Upper Bounds Using Integer Programming}\label{sec:larger}
Let $G=(E^{-1})^TD$, where $E$ is the  $E$-transform of $\E(r)$ and
$a=f^TG^{-1}$. Notice that the condition $Gz \geq 0$ is implied by the
condition $x\geq0$ when $z=G^{-1}x$. The function $I(K_k)+I({\overline K_k})$
in the new coordinates is $f^Tz=a\cdot x$. It turns out that $a_i \geq
0$. This is because the orthogonal parameters $x$ are always $\geq 0$ since
$x_i$ is the number of $r$-subsets containing a graph isomorphic to $g_i$ and
clearly the Ramsey invariant must be satisfy $\forall i: a \cdot e_i \geq 0$, where
$e_i$ is the elementary unit vector. This helps to reduce the number of free
variables significantly. More precisely the number of non-zero components
$a_i$ equals the number of graphs $g \in \E(r)$
s.t. $I(K_k)(g)+I(\overline{K_k}) >0$.

Thus in order to find a zero of $a\cdot x$ we must have $x_i=0$ for all $i$ s.t. $a_i\neq 0$. Let $\Z$ be the set of indices s.t. $a_i \neq 0$. Let $G^*$ be the minor of $G$ containing all the rows of $G$ with indices in $\Z$. In the original coordinates $z$ the condition $x_i=0,i\in \Z$ amounts to
\begin{equation}
G^*z=0.
\end{equation}
This helps us to reduce the number of free variables significantly when most of
$\E(r)$ belongs to $\mathcal{Z}$ i.e. for moderate size $r$, $k < r < n$.

To handle the non-linear constraints we linearize the nonlinear constraints by assigning all the
possible values of the minimal set of graph invariants whose terms appear in
every nonlinear term. In $\E(r)$ we must loop over all the possible values of
invariants $I(g_i)$ s.t. $cv(g_i)\leq \lfloor r/2 \rfloor$ because if the
product $I(g_i)I(g_j)$ satisfies $cv(g_i)+cv(g_j)\leq r$ then the other one,
say $cv(g_i)$, must be smaller or equalt to $\lfloor r/2 \rfloor$. Since at the minimum of the Ramsey-invariant
$z_1=\lfloor{n \choose 2}/2 \rfloor$ and $z_0=1$ the number of free variables
is smaller. We will call the variables $z_i$ s.t. $cv(z_i)\leq \lfloor r/2 \rfloor$ \emph{the assigned variables}.

Let $G=(E^{-1})^TD$, where $E$ is the $E$-transform of $\E(r)$. Let $a=f^TD^{-1}E^T$ and $\Z$ the set of indices s.t. $a_i \neq 0$. Let $G^*$ be the minor of $G$ containing all the rows of $G$ with indices in $\Z$. Let 
\begin{equation}
K=\left [ \begin{array}{c}
G^* \\
P\end{array} \right],
\end{equation}
where $P$ is the matrix of the linearized nonlinear constraints i.e. it contains
rows giving the dependencies $z_iz_j-\sum_k c_{ij}^k z_k=0$, where at least
one of the $z_i$ or $z_j$ is assigned. Explicitly if the $z_i$ is the assigned
variable but $z_j$ is not, then add the value of $z_i$ in the column
corresponding to $z_j$ on the same row. Let $p$ be a column vector of the same height as $K$, initially set to zero. If both variables are assigned then put the product $z_iz_j$ in the $k^{\mathrm{th}}$ element of $p$, where $k$ is the index of the row corresponding to the product $z_iz_j$ in $K$.

Rearrange the columns of $K$ so that the assigned variables are in columns
from $0$ to $m-1$. Let $\pi$ denote this permutation matrix acting from the right.

The constraint $Kz=0$ becomes $H^*z^*+h^*=0$, where $H^*$ is the minor of $K$ not containing the columns corresponding to the assigned variables and $h^*=p+\sum_{i=0}^{m-1} K_iz_i$, where $K_i$ denote the $i^{\mathrm{th}}$ column of the matrix $K$.

Once we find a solution $z^0$ to $H^*z^0+h^*=0$, all the solutions of $Kz=0$ are given by $z^*=[z^a;z^0+ker(H^*)]$, where $z^a$ is the vector of assigned variables and $ker(H^*)$ is the kernel of the matrix $H^*$. We denote by $V$ the $\mathbb{Z}$-basis of the kernel i.e. $ker(H^*)=Vx,x\in \mathbb{Z}^r$ for some $r$. See \cite{Cohen} how to calculate it.

Let $G^l$ be the remaining rows of the matrix $G$ after the rows in $G^*$ have been erased and the columns have been rearranged. Let $H^l$ be the minor of $G^l$ not containing the columns corresponding to the assigned variables and $h^l=\sum_{i=0}^{m-1} G_i^lz_i$, where $G_i^l$ is the $i^{\mathrm{th}}$ column of the matrix $G^l$.
\begin{proposition}
All the $r$-graphic solutions of $Gz\geq 0$ s.t. $f^Tz=0$ and $z\geq 0$ are given by the inequality
\begin{equation}
\left [ \begin{array}{c}
I \\
H^l\end{array} \right]\left( z^0+Vc  \right)+\left [ \begin{array}{c}
0 \\
h^l\end{array} \right] \geq 0
\end{equation}
in the new coordinates $c$. The solution $z$ in the original coordinates is $z=\pi[z^a;z^0+Vc]$, where $z^a$ is the vector of assigned variables and $\pi$ is the permutation used to rearrange the columns of the problem.
\end{proposition}
To restrict the loops over assigned variables we use the lower and upper bounds described in Section \ref{sec:loop} which are compactly denoted here by $L_j \leq z_j \leq U_j$.

Now we are ready to describe the algorithm for computing the possible upper bound for the Ramsey numbers implied by the $G$-poset $\E(r)$. The algorithm prints all the zeros of the Ramsey invariant which are $r$-graphic.

The input matrices $G^l$ and $G^*$ have been rearranged with the permutation
$\pi$. The algorithm uses subroutines MLLL and Inverse\_Image in \cite{Cohen}.
\begin{table}[!htk]
\fbox{\hspace{0.2cm}\begin{minipage}[h]{13cm}
\begin{itemize}
\item[] \textsc{r-graphic zeros of the Ramsey Invariant}
\item[Input:] $n$,$G^l$,$G^*$,$L_k$,$U_k$, $\pi$
\item[1] Set $z_0=1$ and $z_1=\lfloor{n \choose 2}/2 \rfloor$ 
\item[2] Loop through all the remaining assigned variables s.t. $L_j \leq z_j \leq U_j$ for $j=2\ldots m-1$.
\subitem 3 Calculate $P$ and $p$, set $K=\left [ \begin{array}{c}G^* \\P\end{array} \right]$.
\subitem 4 Calculate $H^l$ and $H^*$.
\subitem 5 Calculate the $\mathbb{Z}$-basis $V$ of $ker(H^*)$ by MLLL-algorithm.
\subitem 6 Calculate $p$, $h^*=p+\sum_{i=0}^{m-1} K_iz_i$ and $h^l=\sum_{i=0}^{m-1} G_i^lz_i$.
\subitem 7 Calculate $z^0=$Inverse\_Image$(H^*,h^*)$ i.e. find $z^0$ s.t. $H^*z^0+h^*=0$.
\subitem 8 Set $h=\left [ \begin{array}{c}I \\H^l\end{array} \right]z^0+\left [ \begin{array}{c}0 \\h^l\end{array} \right]$ and $B=\left [ \begin{array}{c}V \\H^lV \end{array} \right]$.
\subitem 9 Find all solutions to $Bc+h\geq 0$. If solutions exits, print(``Solutions found $z=\pi[z^a;z^0+Vc]$.'') for all solutions $c$.
\item[10] If there were no solutions, print(``$r(k)\leq n$'');
\end{itemize}
\end{minipage}\hspace{0.2cm}}
\end{table}

This algorithm essentially transforms the Ramsey-problem into several integer
polyhedron problems. There is a huge literature of papers discussing how to
find the integer points in the polyhedron. There is for instance a polynomial
time algorithm for counting the number of integer points in polyhedra when the
dimension is fixed due to A. I. Barvinok \cite{int1}. The number of free
variables in the problem is roughly equal to the size of the set
\begin{equation}
\{g \in \E(r)| I(K_k)(g)+I(\overline{K_k})(g)=0\}.
\end{equation}

We found that $\E(5)$ is too weak for finding an upper bound for $r(4)$. There
are interior points in the polyhedron given by the constraints.
One solution with $n=18$ and $z_1=76$ is
\begin{eqnarray}
z^T=[1,76,0,2850,144943,40,49161,2559,23294,82,15,121162,\\ \nonumber
41864,9033,104781,107484,77509,3789,1,89219,237949,324866,\\ \nonumber
16203,27998,78056,0,540733,95195,0,3,0,70440,0,0].
\end{eqnarray}
This vector is not graphic since we know that $r(4)=18$ but it is $5$-graphic.
There are $156$ variables/graphs in $\E(6)$ making this approach
unattractive. There are, however, invariants which are sufficient for cliques
and do not grow exponentially on $n$.  

\subsection{Clique-Theoretic Newton Relations}\label{sec:newton}
We notice that the coefficients of basic graph invariants in Proposition
\ref{pro:anti} depend only on the parameters $e=|A|$ and $v=cv(A)$ and thus we
may write the result as follows.
\begin{proposition}\label{the:T3}
We have
\begin{equation}
I({\overline K}_k)={n \choose k}+\sum_{v=2}^{k} \sum_{e=\lceil v/2 \rceil}^{{v \choose 2}} (-1)^e {n-v \choose k-v}\sigma_e^v,
\end{equation}
where 
\begin{equation}
\sigma_e^v({\mathcal G})=\sum_{A \subseteq {\mathcal G},cv(A)=v,|A|=e} I(A)({\mathcal G}).
\end{equation}
Secondly 
\begin{equation}
I(K_k)=\sigma_{{k \choose 2}}^k.
\end{equation}
\end{proposition}

The parameters $\sigma_e^v$ have a similar relationship with the power sums 
\begin{equation}
h_e^v(\mathcal{G}) := \sum_{A \subseteq \mathcal{G},v(A)=v} |A|^e 
\end{equation}
as the classical elementary symmetric polynomials $\sigma_e$ have with the power sums $h_e$. These relations are linear when the variables $a_{ij}\in \{0,1\}$.
\begin{theorem}\label{the:T4}
When $a_{ij}\in \{0,1\}$ the parameters $\sigma_e^v$ and $h_e^v$ are in linear correspondence
\begin{equation}
\sigma_e^v=\sum_{w=\omega(e)}^{v} \sum_{f=1}^{e} (-1)^{e-f+v-w}\frac{\sigma_{e-f}([e-1]) \Pi_{j=1}^{v-w} n-v+j}{e!(v-w)!}h_f^w,
\end{equation}
where $\omega(x)=\lceil(1/2+\sqrt{1+8x}/2\rceil$ and
$\sigma_a([b])=\sigma_a(1,2,\ldots,b)$, where $\sigma_a$ is the classical
elementary symmetric polynomial.
\end{theorem}
\begin{proof}
We use binomial sums $b_e^v({\mathcal G})=\sum_{v(A)=v,A\subseteq{\mathcal G}} {|A| \choose e}$ as mediators. The connection to $\sigma_e^v$ is given by
\begin{equation}\label{eq:eq7}
\sigma_e^v=b_e^v-\sum_{w=\omega(e)}^{v-1} {n-w \choose v-w} \sigma_e^w.
\end{equation}
This result follows from splitting the binomial sum in parts s.t. the number of vertices connected to the edges is $w$. For each such subgraph the remaining vertices can be selected in ${n-w \choose v-w}$ different ways. The sum is over all possible values of $(e,w)$ s.t. there exists a graph with those parameters.
\begin{lemma}\label{the:T5}
We can use the equation (\ref{eq:eq7}) to solve
\begin{equation}
\sigma_e^v= \sum_{w=\omega(e)}^{v} d_{v-w}^{v}b_e^w,
\end{equation}
where
\begin{equation}\label{lala}
d_i^v=\frac{(-1)^i}{i!} \prod_{j=1}^{i}(n-v+j).
\end{equation}
\end{lemma}
\begin{proof}
First notice $\sigma_e^{\omega(e)}=b_e^{\omega(e)}$ by (\ref{eq:eq7}) which means that
$d_0^{\omega(e)}=1$. Also $d_0^v=1$ according to (\ref{eq:eq7}). Write then
\begin{equation}
\sigma_e^{\omega(e)+k}= d_0^{\omega(e)+k} b_e^{\omega(e)+k} + d_1^{\omega(e)+k} b_e^{\omega(e)+k-1} + \cdots + d_{k}^{\omega(e)+k}b_e^{\omega(e)}
\end{equation}
and compare the coefficients with the expanded (\ref{eq:eq7}):
\begin{eqnarray}
\sigma_e^{\omega(e)+k}&=&b_e^{\omega(e)+k} -(n-\omega(e)-k+1)\sigma_e^{\omega(e)+k-1} \\ \nonumber
&-&{n-\omega(e)-k+2 \choose 2}
\sigma_e^{\omega(e)+k-2}- \cdots -{n-\omega(e) \choose k}\sigma_e^{\omega(e)}.
\end{eqnarray}
The coefficients $d_i^v$ satisfy the recursion
\begin{equation}
d_i^v=-\sum_{j=1}^i {n-v+j \choose j}d_{i-j}^{v-j}.
\end{equation}
The solution to this recursion is (\ref{lala}) with the initial values
$d_0^v=1$, which can be seen by the substitution
\begin{eqnarray}
-\sum_{j=1}^i {n-v+j \choose j}\frac{(-1)^{i-j}}{(i-j)!}
 \prod_{k=1}^{i-j}(n-v+j+k) \\ \nonumber
=\sum_{j=1}^i (-1)^{i-j+1} \frac{(n-v+j)!}{(n-v)!j! (i-j)!} \prod_{k=1}^{i-j}
 (n-v+j+k) \\ \nonumber
=\sum_{j=1}^i (-1)^{i-j+1}\frac{(n-v+i)!}{(n-v)!j! (i-j)!} \\ \nonumber
=\prod_{k=1}^i (n-v+k) \sum_{j=1}^i(-1)^{i-j+1} \frac{(n-v)!}{(n-v)!j! (i-j)!}
 \\ \nonumber
=\prod_{k=1}^i (n-v+k) \sum_{j=1}^i(-1)^{i-j+1} \frac{1}{j! (i-j)!} \\  \nonumber
= \prod_{k=1}^i (n-v+k) \sum_{j=1}^i(-1)^{i-j+1} {i \choose j}/i! \\ \nonumber
= \frac{(-1)^i}{i!} \prod_{j=1}^{i}(n-v+j),
\end{eqnarray}
where we used  $(-1)^i=\sum_{j=1}^i(-1)^{i-j+1} {i \choose j}$ which follows
by similar reasoning to Lemma \ref{the:apu}.
\end{proof}

Next we express the $b_e^v$ by using power sums. 

\begin{lemma}\label{the:T6}
\begin{equation}
b_e^v=\sum_{f=1}^e (-1)^{e-f} \frac{\sigma_{e-f}([e-1])}{e!}h_f^v.
\end{equation}
\end{lemma}
\begin{proof}
Since 
\begin{equation}
{x \choose e}=\sum_{f=1}^e (-1)^{e-f} \frac{\sigma_{e-f}([e-1])}{e!}x^f,
\end{equation}
we may substitute this in
\begin{eqnarray}
b_e^v=\sum_{V(A)=v}{|A| \choose e}=\sum_{V(A)=v} (-1)^{e-f} \frac{\sigma_{e-f}([e-1])}{e!}|A|^f \\ \nonumber
=\sum_{f=1}^e (-1)^{e-f} \frac{\sigma_{e-f}([e-1])}{e!}h_f^v.
\end{eqnarray}
and obtain the result.
\end{proof}

We remark that $\sigma_a([b])$ can be computed recursively
\begin{equation}
\sigma_a([b])=b\sigma_{a-1}([b-1])+\sigma_a([b-1]),
\end{equation}
where $\forall b:\sigma_0([b])=1$ and $\forall a>b:\sigma_a([b])=0$.

Finally we combine Lemmas \ref{the:T5} and \ref{the:T6} to obtain the result.
\end{proof}

\begin{example}
\begin{eqnarray}
I(\overline{K_3})= & {n\choose 3}-\left (n-2\right )h_{{1,2}}-5/6\,h_{{1,3}}+h_{{2,3}}-1/6\,h_{{3,3}} \\ \nonumber
I(\overline{K_4})= & {n\choose 4}-{n-2\choose 2}h_{{1,2}}-{\frac {29}{20}}\,h_{{1,4}}+{\frac {203}{90}}\,h_{{2,4}}-{\frac {49}{48}}\,h_{{3,4}} \\ \nonumber
 & +{\frac {35}{144}}\,h_{{4,4}}-{\frac {7}{240}}\,h_{{5,4}}+{\frac {1}{720}}\,h_{{6,4}} \\ \nonumber
I(\overline{K_5})= & {\frac {1}{120}}\,{n}^{5}-1/12\,{n}^{4}-1/6\,{n}^{3}h_{{1,2}}+{\frac {7}{24}}\,{n}^{3}+3/2\,{n}^{2}h_{{1,2}}\\ \nonumber
 & +1/4\,{n}^{2}h_{{1,3}}-1/4\,{n}^{2}h_{{2,3}}-{\frac {5}{12}}\,{n}^{2}-13/3\,nh_{{1,2}}-7/4\,nh_{{1,3}} \\ \nonumber
 & +7/4\,nh_{{2,3}}-1/2\,nh_{{1,4}}+1/2\,nh_{{2,4}}+1/5\,n+4\,h_{{1,2}}+3\,h_{{1,3}}-3\,h_{{2,3}}+2\,h_{{1,4}} \\ \nonumber
 & -2\,h_{{2,4}}-{\frac {3601}{2520}}\,h_{{1,5}}+{\frac {151933}{50400}}\,h_{{2,5}}-{\frac {84095}{36288}}\,h_{{3,5}} \\ \nonumber
 & +{\frac {341693}{362880}}\,h_{{4,5}}-{\frac {8591}{34560}}\,h_{{5,5}}+{\frac {7513}{172800}}\,h_{{6,5}} \\ \nonumber
 & -{\frac {121}{24192}}\,h_{{7,5}}+{\frac {11}{30240}}\,h_{{8,5}}-{\frac {11}{725760}}\,h_{{9,5}}+{\frac {1}{3628800}}\,h_{{10,5}}
\end{eqnarray}
\end{example}

\subsection{Syzygies for Symmetric Polynomials}
Do the parameters $\sigma_e^v$ satisfy some algebraic relations? This is non-trivial since the product of $\sigma_{e_1}^{v_1}\sigma_{e_2}^{v_2}$ is not closed in the set of symmetric polynomials $\sigma_e^v$. However by computer search in $\E(8)$ we were able to find the following dependencies. These are actually the only ones in $\E(8)$ and smaller except for the equivalent dependencies were the leading monomials are the same but the remaining monomials vary.
\begin{theorem}
The parameters $\sigma_e^v$ satisfy at least the following general syzygies:
\begin{tiny}
\begin{eqnarray}
-\sigma_1^2\sigma_1^2\sigma_1^2\sigma_1^2+2\sigma_2^4\sigma_1^2\sigma_1^2+2\sigma_2^3\sigma_1^2\sigma_1^2+\sigma_1^2\sigma_1^2\sigma_1^2=0\\
-\sigma_6^4\sigma_1^2\sigma_1^2+2\sigma_6^4\sigma_2^4+2\sigma_6^4\sigma_2^3+\sigma_6^4\sigma_1^2=0\\
-\sigma_5^4\sigma_1^2\sigma_1^2+2\sigma_5^4\sigma_2^4+2\sigma_5^4\sigma_2^3+\sigma_5^4\sigma_1^2=0\\
-\sigma_4^4\sigma_1^2\sigma_1^2+2\sigma_4^4\sigma_2^4+2\sigma_4^4\sigma_2^3+\sigma_4^4\sigma_1^2=0\\
-\sigma_3^4\sigma_1^2\sigma_1^2-1/3\sigma_2^4\sigma_1^2\sigma_1^2-\sigma_3^3\sigma_1^2\sigma_1^2-1/3\sigma_2^3\sigma_1^2\sigma_1^2+2\sigma_3^4\sigma_2^4+4/3\sigma_2^4\sigma_2^4\\ \nonumber
+2\sigma_2^4\sigma_3^3+2\sigma_3^4\sigma_2^3+8/3\sigma_2^4\sigma_2^3+2\sigma_3^3\sigma_2^3+4/3\sigma_2^3\sigma_2^3-\sigma_3^6\sigma_1^2-\sigma_3^5\sigma_1^2=0\\
-\sigma_2^4\sigma_1^2\sigma_1^2-3\sigma_3^3\sigma_1^2\sigma_1^2-\sigma_2^3\sigma_1^2\sigma_1^2+4\sigma_2^4\sigma_2^4+6\sigma_2^4\sigma_3^3+8\sigma_2^4\sigma_2^3\\ \nonumber
+6\sigma_3^3\sigma_2^3+4\sigma_2^3\sigma_2^3-3\sigma_3^6\sigma_1^2-3\sigma_3^5\sigma_1^2-3\sigma_3^4\sigma_1^2=0\\
-\sigma_3^3\sigma_1^2\sigma_1^2-1/3\sigma_2^3\sigma_1^2\sigma_1^2+2/3\sigma_2^4\sigma_2^4+2\sigma_2^4\sigma_3^3+2\sigma_2^4\sigma_2^3+2\sigma_3^3\sigma_2^3\\ \nonumber
+4/3\sigma_2^3\sigma_2^3-\sigma_3^6\sigma_1^2-\sigma_3^5\sigma_1^2-\sigma_3^4\sigma_1^2-1/3\sigma_2^4\sigma_1^2=0\\
-\sigma_2^3\sigma_1^2\sigma_1^2+2\sigma_2^4\sigma_2^4+6\sigma_2^4\sigma_2^3+4\sigma_2^3\sigma_2^3-3\sigma_3^6\sigma_1^2-3\sigma_3^5\sigma_1^2-3\sigma_3^4\sigma_1^2-\sigma_2^4\sigma_1^2-3\sigma_3^3\sigma_1^2=0\\
-\sigma_1^2\sigma_1^2\sigma_1^2+4\sigma_2^4\sigma_2^4+8\sigma_2^4\sigma_2^3+4\sigma_2^3\sigma_2^3-6\sigma_3^6\sigma_1^2\\ \nonumber
-6\sigma_3^5\sigma_1^2-6\sigma_3^4\sigma_1^2-6\sigma_3^3\sigma_1^2+\sigma_1^2\sigma_1^2=0\\
-\sigma_2^4\sigma_2^4-2\sigma_2^4\sigma_2^3-\sigma_2^3\sigma_2^3+3/2\sigma_3^6\sigma_1^2+3/2\sigma_3^5\sigma_1^2\\ \nonumber
+3/2\sigma_3^4\sigma_1^2+1/2\sigma_2^4\sigma_1^2+3/2\sigma_3^3\sigma_1^2+1/2\sigma_2^3\sigma_1^2=0\\
-\sigma_3^6\sigma_1^2-\sigma_3^5\sigma_1^2-\sigma_3^4\sigma_1^2-\sigma_3^3\sigma_1^2+4\sigma_4^8+4\sigma_4^7+4\sigma_4^6+3\sigma_3^6+4\sigma_4^5+3\sigma_3^5+4\sigma_4^4+3\sigma_3^4+3\sigma_3^3=0\\
-\sigma_2^4\sigma_1^2-\sigma_2^3\sigma_1^2+3\sigma_3^6+3\sigma_3^5+3\sigma_3^4+2\sigma_2^4+3\sigma_3^3+2\sigma_2^3=0\\
-\sigma_1^2\sigma_1^2+2\sigma_2^4+2\sigma_2^3+\sigma_1^2=0.
\end{eqnarray}
\end{tiny}
\end{theorem}
Since most of the products  $\sigma_{e_1}^{v_1}\sigma_{e_2}^{v_2}$ are not
closed under multiplication in the parameters $\sigma_e^v$, the Ramsey problem
for instance cannot be necessarily solved purely in terms of these
parameters. Moreover in order to calculate the required syzygies for $r(5)$ we
need calculations in $\E(10)$ to ensure that the products $\sigma_{e_1}^5\sigma_{e_2}^5$ are covered. 
Our implementation of the required algorithms seems to consume several gigabytes of memory thus exceeding today's desktop computers' capabilities.

\section{Open Problems}
We have shown that $\E(4)$ is strong enough to prove $r(3)\leq 6$ and $\E(5)$ is too weak for proving $r(4)\leq 18$. This leads to the following questions.

\newtheorem{op1}{Problems}[section]
\begin{op1}
Solve the following questions:
\begin{itemize}
\item[i] How large an $G$-poset $\E(r)$ is required to find an upper bound for $r(k)$?
\item[ii] Can you solve $r(4)$ and perhaps $r(5)$ by utilizing the local
  parameters $z(i)$, $z(i_1,i_2)$, $\ldots$, $z(i_1,\ldots,i_k)$ to give stronger constraints?
\item[iii] By using Theorem \ref{the:graafisuus1}, is it possible to find new lower
  bounds to Ramsey numbers by reconstruction?
\end{itemize}
\end{op1}
The parameters in the Ramsey problem can be reduced to polynomial size
by introducing the $\sigma_{e}^v$-polynomials. However the structure of the inequalities
is not easy. Let $E_{\sigma}$ be the evaluated values for the vector
\begin{equation}
\sigma=[\sigma_{1}^2,\sigma_{2}^{3},\sigma_3^3,\ldots,\sigma_{{r \choose 2}}^r]
\end{equation}
over the $G$-poset $\E(r)$. Then similarly to Proposition \ref{pro:kat}, we have 
\begin{equation}
\forall c \in \mathbb{R}^{m} \ \mathrm{s.t.} \  E_\sigma c \leq 0 : c^T D
\sigma \leq 0.
\end{equation}
Since $E_\sigma$ is not a square matrix, we cannot find easily characterization to
weakly graphic vectors $\sigma$.

\newtheorem{op2}[op1]{Problem}
\begin{op2}
Find inequalities and general syzygies for the $\sigma_e^v$ parameters or
alternatively for the power sums $h_e^v$.
\end{op2}

\end{document}